\documentclass[12pt,reqno]{amsart}

\usepackage[utf8]{inputenc}
\usepackage[T1]{fontenc}
\usepackage{lmodern}
\usepackage{amsmath, amssymb, amsthm,polynom}
\usepackage{multirow}
\usepackage{subcaption}
\usepackage{fourier} 
\usepackage{dsfont}
\usepackage{mathrsfs}
\usepackage{mathtools}
\usepackage{bm}
\usepackage[authoryear,round]{natbib} 
\usepackage[hidelinks]{hyperref} 
\usepackage{xcolor} 
\hypersetup{hidelinks}
\usepackage{enumitem}
\usepackage{tikz}
\usepackage{graphicx}
\usepackage{cleveref}
\usepackage{booktabs}
\usepackage{float}
\usepackage[
letterpaper,
margin=1in,
includeheadfoot,
heightrounded
]{geometry}
\theoremstyle{remark}
\newtheorem{remark}{Remark}
\numberwithin{equation}{section}

\newcommand{\aver}[1]{\left\{\!\!\left\{#1\right\}\!\!\right\}}

\usepackage{stmaryrd}
\newcommand{\jump}[1]{\left\llbracket#1\right\rrbracket} 

\newcommand{\commentout}[1]{{}} 

\newtheorem{corollary}{Corollary}[section]

\newtheorem{theorem}{Theorem}[section]
\newtheorem{lemma}{Lemma}[section]

\newcommand{\norm}[1]{\left\|#1\right\|}
\newcommand{\normtvb}[1]{{\left\vert\kern-0.25ex\left\vert\kern-0.25ex\left\vert #1
    \right\vert\kern-0.25ex\right\vert\kern-0.25ex\right\vert}}

\newcommand{\abs}[1]{\left|#1\right|}

\newcommand{\bfg}{{\bf g}}

\newcommand{\bfn}{{\bf n}}

\newcommand{\bftau}{\boldsymbol{\tau}}

\newif\ifyhmark
\yhmarkfalse
\newcommand{\yhred}[1]{\ifyhmark\textcolor{red}{#1}\else#1\fi}

\definecolor{GCRevisionBlue}{RGB}{0,65,180}
\newif\ifgcrevisions
\gcrevisionstrue

\title{Geometry-Conforming Finite Element Methods for Interface Problems \\
	on Fitted and Unfitted Meshes}

\makeatletter
\newcommand{\affiliations}[1]{\gdef\gc@affiliations{#1}}
\let\gc@affiliations\@empty
\g@addto@macro\@setauthors{%
  \ifx\gc@affiliations\@empty\else
    \vspace{4pt}%
    {\centering\footnotesize\gc@affiliations\par}%
  \fi}
\makeatother

\author{Yuanhui Lin\textsuperscript{1}}
\author{Tao Lin\textsuperscript{2}}
\author{Xu Zhang\textsuperscript{3,*}}
\author{Minfu Feng\textsuperscript{1}}

\affiliations{%
\textsuperscript{1}School of Mathematics, Sichuan University, Chengdu, Sichuan 610065, China
(\href{mailto:linyuanhui001@stu.scu.edu.cn}{linyuanhui001@stu.scu.edu.cn};
\href{mailto:fmf@scu.edu.cn}{fmf@scu.edu.cn})\\
\textsuperscript{2}Department of Mathematics, Virginia Tech, Blacksburg, VA 24060, USA
(\href{mailto:tlin@vt.edu}{tlin@vt.edu})\\
\textsuperscript{3}Department of Mathematics, Oklahoma State University, Stillwater, OK 74078, USA
(\href{mailto:xzhang@okstate.edu}{xzhang@okstate.edu})\\
\textsuperscript{*}Corresponding author}


\AtBeginDocument{\colorlet{red}{black}}
\begin{document}
\pagestyle{plain}
\setlength{\emergencystretch}{1em}
	
\begin{abstract}
We develop an arbitrary-degree geometry-conforming finite element (GC-FE)
framework for two-dimensional elliptic boundary value and interface problems
on curved domains. Using the Frenet--Serret transformation, curved-boundary
and interface segments are represented exactly, while polynomials in
Frenet coordinates generate generally nonpolynomial local shape functions in
physical coordinates. To solve the interface problems on fitted meshes, we adopt a symmetric interior penalty discontinuous Galerkin (SIPDG) method with harmonic flux averages and penalty weights. We prove well-posedness and an optimal weighted energy error estimate with constants independent of both coefficient values and their ratio. Optimal $L^2$
convergence follows under a weighted dual regularity assumption. By retaining the prescribed curves exactly, the method avoids the geometric variational crime associated with curved-geometry approximation and requires no corresponding geometric
consistency estimates. GC-FE spaces can also be combined with geometry-conforming immersed finite element (GC-IFE) spaces to solve interface problems on interface-unfitted meshes. 
Numerical examples are provided to demonstrate features of
the proposed GC-FE method.
\end{abstract}
	
	\maketitle
\section{Introduction}
\label{sec:intro}

Let $\Omega \subset \mathbb{R}^2$ be a bounded domain with a sufficiently smooth curved boundary
$\partial\Omega$, and consider the second-order elliptic boundary value problem
\begin{align}
  -\nabla\cdot(\beta\,\nabla u) &= f \quad \text{ in } \Omega, \label{eq:bvp-pde} \\
  u &= g \quad \text{ on } \partial\Omega . \label{eq:bvp-bc}
\end{align}
We also consider the closely related \emph{elliptic interface problem}. Let
$\Gamma_i$ be a smooth curve that divides $\Omega$ into two subdomains
$\Omega^{+}$ and $\Omega^{-}$, and let $\beta$ be uniformly positive and
piecewise smooth, with a possible jump across $\Gamma_i$. We seek $u$ satisfying
\begin{align}
  -\nabla\cdot(\beta\,\nabla u) &= f \quad \text{in } \Omega^{+}\cup\Omega^{-}, \label{eq:interface-pde} \\
  u &= g \quad \text{on } \partial\Omega, \label{eq:interface-bc}
\end{align}
together with the interface jump conditions
\begin{equation}
\label{eq:jumps}
\jump{u}_{\Gamma_i}=0,
\qquad
\jump{\beta\,\partial_{\mathbf{n}}u}_{\Gamma_i}=0,
\end{equation}
where $\mathbf{n}$ is a unit normal to $\Gamma_i$ and
$\jump{\cdot}_{\Gamma_i}$ denotes the jump across the interface. Boundary value
problems of the form \eqref{eq:bvp-pde}--\eqref{eq:bvp-bc} and interface
problems of the form \eqref{eq:interface-pde}--\eqref{eq:jumps} arise widely
in science and engineering, including heat conduction, porous-media flow,
electromagnetics, and the modeling of composite and multiphase materials.
In these applications, the curved boundary $\partial\Omega$ or interface
$\Gamma_i$ carries essential physical and geometric information.

For both problems, high-order accuracy depends on how the curved geometry is represented. When a curved domain is
approximated by a straight-edged mesh, the polygonal boundary generally differs
from the physical boundary by $O(h^2)$. For high-order finite element spaces,
the resulting geometric consistency error may dominate the discretization
error unless the geometry is approximated to a correspondingly high order.
Such a domain perturbation is a classical example of a \emph{variational
crime} in the sense of Strang \citep[see][]{1973StrangFix,1978Ciarlet}. An analogous
issue arises in interface problems, where the jump conditions
\eqref{eq:jumps} should ideally be imposed on the physical interface $\Gamma_i$
rather than on a polygonal approximation. Accurate treatment of curved
boundaries and interfaces is therefore essential for the analysis and
implementation of high-order finite element methods.

For interface problems, the geometric challenge is at least as significant.
An interface-fitted mesh must conform either to the physical interface or to
a sufficiently accurate approximation, and changes in the interface geometry
may require costly remeshing. Immersed finite element (IFE) methods alleviate
these difficulties by using meshes independent of the interface and modifying
the local finite element functions on interface elements to incorporate the
jump conditions \eqref{eq:jumps}, while retaining standard polynomial spaces
on non-interface elements. Since the early work
on IFE methods
\citep{Li1998,LiLinWu2003,2008HeLinLin}, IFE methods have been studied
extensively. Subsequent developments include partially penalized formulations
with optimal-order error estimates \citep{LinLinZhang2015}, higher-degree
constructions based on least-squares procedures \citep{2017AdjeridGuoLin} and
Cauchy extensions \citep{2019GuoLin2}, three-dimensional a~priori analyses
\citep{2020GuoLin,2021GuoZhang}, and extensions to other classes of partial
differential equations
\citep{2013LinSheenZhang,2019AdjeridMoon,2020GuoLinLin,2021ChenZhang,2026ChenZhang}. 
Other unfitted approaches include cut finite element methods \citep[CutFEM;][]{2015BurmanClausHansboLarsonMassing,2025BurmanHansboLarsonZahedi}, which impose interface conditions weakly on cut meshes. Achieving high order with such methods requires a correspondingly accurate interface approximation; \citet{2018LehrenfeldReusken} obtain it with an isoparametric mapping of a piecewise planar reconstruction and prove optimal bounds that include the geometric error. In contrast, the geometry-conforming IFE spaces discussed below incorporate the exact interface.

A classical remedy is the polynomial isoparametric finite element method. In this
approach, a curved computational element $K_h$ is defined as the image
$K_h=F_K(\widehat{K})$ of a reference simplex $\widehat{K}$ under a polynomial
mapping $F_K\in[\mathbb{P}_m(\widehat{K})]^2$. The degree of the geometric
mapping is typically matched to that of the corresponding finite element
space. Detailed treatments can be found in
\citet{1978Ciarlet} and \citet{1973StrangFix}, and the isoparametric framework has also
been extended to high-order discretizations of elliptic interface problems
\citep{LiMelenkWohlmuthZou2010}. Although this approach has a well-developed approximation theory, it generally retains a discrepancy between the physical and computational geometries: the representation is exact only when the prescribed geometry lies in the chosen polynomial mapping space. 

\commentout{
}

Recently, the \emph{Frenet--Serret apparatus} from differential geometry
has been used to represent curved interfaces exactly while retaining
high-order approximation. A tubular neighborhood of the interface is mapped
to a coordinate strip in which the interface is straightened. The jump
conditions can then be imposed on the exact interface, and
geometry-conforming immersed finite element (GC-IFE) spaces of arbitrary
polynomial degree can be constructed systematically
\citep{2024AdjeridLinMeghaichi,2025AdjeridLinMeghaichi}. This construction
incorporates the tangent and normal fields, together with the interface
curvature, directly into the local approximation space. This approach has recently been
extended from Cartesian meshes to unstructured triangular meshes
\citep{2026LinLinZhang}.

The present work extends this use of the Frenet--Serret transformation
(hereafter, Frenet transformation) to elements fitted to curved boundaries
and interfaces. We develop and analyze a \emph{geometry-conforming finite
element (GC-FE)} method for two-dimensional elliptic boundary value and
interface problems. On each curved-boundary or interface-fitted element,
the physical curve forms an exact element edge, while polynomials in Frenet
coordinates generate generally nonpolynomial shape functions in physical
coordinates. For interface-unfitted meshes, GC-FE spaces on curved-boundary
elements are combined with the GC-IFE spaces of \citet{2026LinLinZhang} on
interface-cut elements. This construction retains the physical boundary and
interface exactly in both settings, eliminating the domain- and
interface-perturbation terms that arise in classical polynomial
isoparametric analysis \citep{1972CiarletRaviart,1972CiarletRaviart2,1986Lenoir}.

The GC-FE framework developed in this work shares with isogeometric analysis (IGA), introduced by \citet*{Hughes2005} and further developed by \citet{HughesSangalli2017}, the objective of incorporating the physical geometry exactly and thereby eliminating a geometric approximation error. In standard IGA, the geometry and solution are typically represented using the same spline or NURBS basis functions. The NURBS-enhanced finite element method \citep[NEFEM;][]{sevilla2008nurbs} also retains the exact boundary, but uses polynomial approximation in Cartesian coordinates on boundary elements, together with specially designed quadrature. GC-FE differs from both IGA and NEFEM: the geometry enters through the Frenet transformation of the prescribed curve, so neither the geometry nor the
approximation space requires a spline or NURBS representation; only a
sufficiently smooth parameterization of the curve is needed (of class
$C^{m+2}$ in our analysis). The local spaces consist of polynomials in the
Frenet coordinates. As in IGA and isoparametric finite elements, but unlike
NEFEM, these spaces generally induce nonpolynomial shape functions in
physical coordinates. 

The main contributions of this paper are summarized as follows. 
\begin{enumerate}
\item[(i)] \textbf{Geometry-conforming finite element spaces of arbitrary degree.}
We introduce a systematic construction of arbitrary-degree finite element
spaces on elements adjacent to curved boundaries and fitted interfaces.
The construction retains the exact physical curve as an element edge and,
unlike the classical polynomial isoparametric approach, requires no polynomial surrogate
for the curved geometry.

\item[(ii)] \textbf{A unified treatment of fitted and unfitted 
meshes.}
GC-FE spaces are used on fitted elements with curved interface edges, whereas
GC-IFE spaces \citep{2026LinLinZhang} are used on interface-cut elements of
unfitted triangular meshes, with standard polynomial spaces retained
elsewhere. In both settings, the corresponding geometry-conforming spaces
treat the physical interface exactly.

\item[(iii)] \textbf{Key estimates for the resulting
nonpolynomial spaces.}
We establish optimal-order approximation properties for the GC-FE spaces,
showing that the transformation to nonpolynomial functions in physical
coordinates preserves the optimal approximation capability of the underlying
polynomial spaces in Frenet coordinates. Additionally, we prove inverse and trace inequalities 
for the GC-FE spaces. The constants in these estimates are independent of the
diffusion coefficients for fixed polynomial degree and uniformly regular
meshes and geometry.

\item[(iv)] \textbf{Coefficient-robust discontinuous Galerkin methods.}
We adopt the symmetric interior penalty discontinuous Galerkin method with coefficient-dependent weighted averages \citep{2003Dryja,2006BurmanZunino,2009ErnStephansenZunino,2011CaiYeZhang,2017CaiHeZhang2}. We extend the robust error analysis to the nonpolynomial GC-FE spaces on elements with exact curved geometry, establishing coercivity and an energy error estimate with constants independent of both coefficient values and their ratio.

\end{enumerate}

The remainder of the paper is organized as follows. Section~2 recalls the
Frenet--Serret transformation and establishes the necessary geometric preliminaries.
Section~3 constructs the GC-FE space and derives its approximation properties,
inverse inequalities, and trace inequalities. Section~4 develops
the harmonically weighted SIPDG discretization and proves well-posedness
and coefficient-robust weighted energy estimates on fitted meshes, together
with $L^2$ estimates under a weighted dual regularity assumption. Exact
geometric conformity eliminates the need for a separate geometric-consistency
estimate.  Section~5
presents numerical experiments, including a comparison with the isoparametric
method, coefficient-contrast tests, and a test of the combined GC-FE-GC-IFE
discretization on an interface-unfitted mesh. Section~6 concludes the paper.

	\section{Preliminaries}\label{sec:prelim}
	
	In this section, we recall the Frenet transformation and establish the geometric estimates used in the
construction and analysis of the GC-FE spaces. Some of these results have appeared previously in \citet{2024AdjeridLinMeghaichi,2025AdjeridLinMeghaichi,AdjeridLinMeghaichi2026}; they are included here for completeness.

	\subsection{Geometric setting and mesh classification}
\label{subsec:geometry_mesh}
         Let $\Omega\subset\mathbb R^2$ be a bounded domain with a smooth boundary
$\Gamma_b$. Let $\Gamma_i\subset\Omega$ be a smooth curve that separates
$\Omega$ into two disjoint open subdomains $\Omega^-$ and $\Omega^+$ such that
\[
  \Omega=\Omega^-\cup\Gamma_i\cup\Omega^+.
\]
For $t=b,i$, we assume that $\Gamma_t$ admits a regular $C^2$
parameterization
\begin{equation}
  \mathbf{g}_t(\xi)
  =\begin{bmatrix}g_{t1}(\xi)\\ g_{t2}(\xi)\end{bmatrix},
  \qquad \xi\in[\xi_s,\xi_e].
  \label{eq:parameterization_Gammab_Gamma_i}
\end{equation}

Let $\{\overline{\mathcal{T}}_h\}_{h>0}$ be a family of
shape-regular, straight-sided precursor triangulations associated with
$\Omega$. We say that an edge of an element
$\overline{K}\in\overline{\mathcal{T}}_h$ fits a curve if both endpoints of
the edge lie on the curve. We define
\begin{align*}
  \overline{\mathcal{T}}_h^i
  &:=
  \bigl\{
    \overline{K}\in\overline{\mathcal{T}}_h :
    \overline{K}\text{ is cut by }\Gamma_i
    \text{ and has no edge fitted to }\Gamma_i
  \bigr\}, \\
  \overline{\mathcal{T}}_h^c
  &:=
  \bigl\{
    \overline{K}\in\overline{\mathcal{T}}_h :
    \overline{K}\text{ has a designated edge fitted to }
    \Gamma_b\text{ or }\Gamma_i
  \bigr\}.
\end{align*}
We assume that
$\overline{\mathcal{T}}_h^i\cap\overline{\mathcal{T}}_h^c=\varnothing$
and define
\begin{equation*}
  \overline{\mathcal{T}}_h^r
  :=
  \overline{\mathcal{T}}_h\setminus
  \bigl(
    \overline{\mathcal{T}}_h^i
    \cup
    \overline{\mathcal{T}}_h^c
  \bigr).
\end{equation*}
Thus, every element of $\overline{\mathcal{T}}_h$ belongs to exactly one of
the three classes.

The elements of $\overline{\mathcal{T}}_h^i$ are called
\emph{interface elements}, whereas those of
$\overline{\mathcal{T}}_h^c$ are called \emph{precursor curved elements}. The elements in $\overline{\mathcal{T}}_h^r$ are called \emph{regular elements}. 
We say that
the mesh $\overline{\mathcal{T}}_h$ is \emph{interface-fitted} if
$\overline{\mathcal{T}}_h^i=\emptyset$ and \emph{interface-unfitted}
otherwise. Examples of interface-fitted and interface-unfitted meshes are
shown in Figure~\ref{fig:Mesh_ellipse_circle_fitted_unfitted}.


For each $\overline{K}\in\overline{\mathcal{T}}_h^c$, let
$K=K(\overline{K})$ be the curved triangle obtained by replacing its
designated straight edge with the corresponding arc of $\Gamma_b$ or
$\Gamma_i$. We then set
\begin{equation*}
  \mathcal{T}_h^i
  :=\overline{\mathcal{T}}_h^i,
  \qquad
  \mathcal{T}_h^c
  :=\bigl\{
       K(\overline{K}):
       \overline{K}\in\overline{\mathcal{T}}_h^c
     \bigr\},
  \qquad
  \mathcal{T}_h^r
  :=\overline{\mathcal{T}}_h^r,
\end{equation*}
and define the resulting curved mesh of $\Omega$ by
\begin{equation*}
  \mathcal{T}_h
  :=
  \mathcal{T}_h^i
  \cup
  \mathcal{T}_h^c
  \cup
  \mathcal{T}_h^r.
\end{equation*}

We assume that the curved mesh family
$\{\mathcal{T}_h\}_{h>0}$ is shape-regular in the following sense. For each
$K\in\mathcal{T}_h$, let $h_K:=\operatorname{diam}(K)$ and define its inradius
by
\begin{equation}
  \rho_K
  :=
  \sup\bigl\{
    r>0:
    \text{there exists }X\in K\text{ such that }B(X,r)\subset K
  \bigr\}.
  \label{eq:inradius_of_K}
\end{equation}
We assume that there exists a constant $C_{\mathrm{sh}}>0$, independent of
$h$ and $K$, such that
\begin{equation}
  \frac{h_K}{\rho_K}\leq C_{\mathrm{sh}},
  \qquad \forall ~h>0,\quad K\in\mathcal{T}_h.
  \label{eq:ShapeReg}
\end{equation}
	
	\begin{figure}[h]
		\includegraphics[width = .4\textwidth]{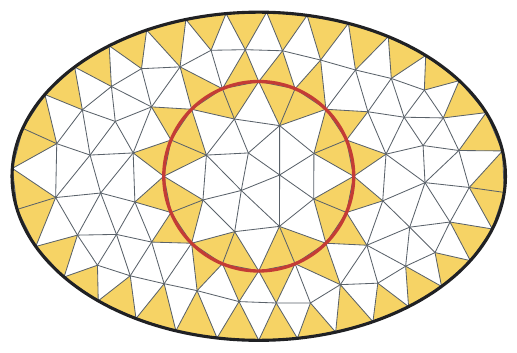}\hspace{0.2in}
		\includegraphics[width = .4\textwidth]{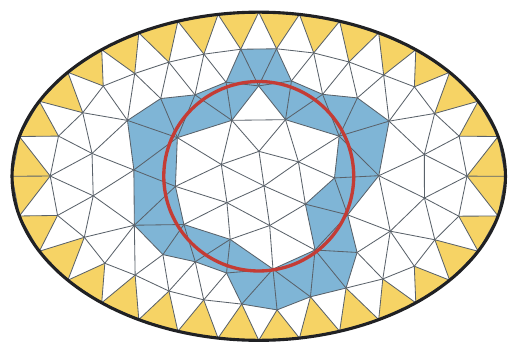}
		\caption{An elliptical domain with a circular interface. Left: an interface-fitted
mesh; right: an interface-unfitted mesh. Regular elements
($\overline{\mathcal{T}}_h^r$), precursor curved elements
($\overline{\mathcal{T}}_h^c$), and interface elements
($\overline{\mathcal{T}}_h^i$) are shown in white, gold, and blue,
respectively. The gray edges outline the precursor triangulations, the thick
dark curve represents the outer boundary $\Gamma_b$, and the red circle
represents the interface $\Gamma_i$.}
		\label{fig:Mesh_ellipse_circle_fitted_unfitted}
	\end{figure}
	
	\subsection{Frenet coordinates and uniform map estimates}
\label{subsec:frenet_coordinates}

Consider a regular $C^2$ curve $\Gamma\subset\mathbb R^2$ parameterized by
\begin{equation}
  \mathbf{g}(\xi)=\begin{bmatrix}g_1(\xi)\\g_2(\xi)\end{bmatrix},
  \qquad \xi\in[\xi_s,\xi_e].
  \label{eq:g}
\end{equation}
	The Frenet–Serret apparatus \citep{gray_2006,ONeill_DiffGeom_2010} associated with the curve $\mathbf{g}(\xi)$ consists of the unit tangent
$\bftau(\xi)$, the unit normal $\bfn(\xi)$, and the curvature
$\kappa(\xi)$:
\begin{equation}
  \bftau(\xi)=\frac{\mathbf g'(\xi)}{\|\mathbf g'(\xi)\|},
  \qquad
  \bfn(\xi)=Q\bftau(\xi),
  \qquad
  \kappa(\xi)
  =\frac{\mathbf g'(\xi)^TQ\mathbf g''(\xi)}
         {\|\mathbf g'(\xi)\|^3},
  \label{eq:Frenet_apparatus}
\end{equation}
where
\[
  Q=\begin{bmatrix}0&1\\-1&0\end{bmatrix}.
\]
The associated Frenet transformation $P_\Gamma: (\eta,\xi) \rightarrow (x,y)$ is
\begin{equation}
  \mathbf x(\eta,\xi)
  =\begin{bmatrix}x(\eta,\xi)\\y(\eta,\xi)\end{bmatrix}
  =P_\Gamma(\eta,\xi)
  :=\mathbf g(\xi)+\eta\bfn(\xi).
  \label{eq:Frenet_Pmap}
\end{equation}

The tubular-neighborhood theorem \citep{2012AbateTovena} implies that, for $\epsilon>0$ sufficiently small, $P_\Gamma$ is a bijection from $[-\epsilon,\epsilon]\times[\xi_s,\xi_e]$ onto the $\epsilon$-tubular neighborhood \[
  N_\Gamma(\epsilon)
  :=P_\Gamma\bigl([ -\epsilon,\epsilon]\times[\xi_s,\xi_e]\bigr).
\]
Its inverse is denoted by
\begin{equation}
  \begin{split}
    R_\Gamma=P_\Gamma^{-1}:N_\Gamma(\epsilon)
      &\longrightarrow[-\epsilon,\epsilon]\times[\xi_s,\xi_e],\\
    \begin{bmatrix}\eta\\\xi\end{bmatrix}
      &=R_\Gamma(x,y)
       =\begin{bmatrix}\eta(x,y)\\\xi(x,y)\end{bmatrix}.
  \end{split}
  \label{eq:Frenet_Rmap}
\end{equation}
When $\Gamma$ is closed, $\mathbf g$ is extended periodically with period
   $\xi_e-\xi_s$, the bijectivity of $P_\Gamma$ is understood modulo the period,
   and on each $K\in\mathcal T_h^c\cup\mathcal T_h^i$ we use the branch of $\xi$
   that is continuous on $K$.
A direct calculation gives
\begin{align}
  \nabla\eta(x,y)
  &=\bfn(\xi),
  \label{eq:grad_eta}\\
  \nabla\xi(x,y)
  &=\norm{\bfg'(\xi)}^{-1}
    \bigl(1+\eta\kappa(\xi)\bigr)^{-1}\bftau(\xi),
  \label{eq:grad_xi}\\
  DP_\Gamma(\eta,\xi)
  &=\Bigl[\bfn(\xi),
     \norm{\bfg'(\xi)}\bigl(1+\eta\kappa(\xi)\bigr)
     \bftau(\xi)\Bigr],
  \label{eq:J_of_P}\\
  \det DP_\Gamma(\eta,\xi)
  &=\norm{\bfg'(\xi)}\bigl(1+\eta\kappa(\xi)\bigr),
  \label{eq:det_J_of_P}\\
  DR_\Gamma(x,y)
  &=\begin{bmatrix}
      \bfn^T(\xi)\\
      \norm{\bfg'(\xi)}^{-1}
      \bigl(1+\eta\kappa(\xi)\bigr)^{-1}\bftau^T(\xi)
    \end{bmatrix},
  \label{eq:J_of_Pinv}\\
  \det DR_\Gamma(x,y)
  &=\norm{\bfg'(\xi)}^{-1}
    \bigl(1+\eta\kappa(\xi)\bigr)^{-1}.
  \label{eq:det_J_of_Pinv}
\end{align}
Here and below, $(x,y)$ and $(\eta,\xi)$ are related through
\eqref{eq:Frenet_Pmap} and \eqref{eq:Frenet_Rmap}.
Since $\Gamma$ is regular and $C^2$, there exist positive constants $C_1, C_2$, and $C_3$ such that
\begin{equation}
  0<C_1\le\norm{\mathbf g'(\xi)}\le C_2,
  \qquad
  \norm{\mathbf g''(\xi)}\le C_2,
  \qquad
  \abs{\kappa(\xi)}\le C_3,
  \quad \xi\in[\xi_s,\xi_e].
  \label{eq:bnds_for_g'_kappa}
\end{equation}
The following lemma states uniform bounds for the Frenet transformations and their Jacobians.

\begin{lemma}\label{lem:Frenet_uniform_bounds}
There exist constants $\epsilon_0>0$ and $\tilde{C}_J, ~C_J > 0$ such that, for every
$0<\epsilon\le\epsilon_0$,
\begin{align}
  \frac12\le 1+\eta\kappa(\xi)\le 2,
  &\qquad
  (\eta,\xi)\in[-\epsilon,\epsilon]\times[\xi_s,\xi_e],
  \label{eq:1+eta_kappa_bnd}\\
  \tilde{C}_J\le\abs{\det DP_\Gamma(\eta,\xi)}\le C_J,
  &\qquad
  \norm{DP_\Gamma(\eta,\xi)}\le C_J,
  \quad  (\eta,\xi)\in[-\epsilon,\epsilon]\times[\xi_s,\xi_e],
  \label{eq:det_J_of_P_bnd}\\
  \tilde{C}_J\le\abs{\det DR_\Gamma(x,y)}\le C_J,
  &\qquad
  \norm{DR_\Gamma(x,y)}\le C_J,
  \quad (x,y)\in N_\Gamma(\epsilon).
  \label{eq:det_J_of_Pinv_bnd}
\end{align}
\end{lemma}
	
\begin{proof}
Choose $\epsilon_0>0$ so that
$\epsilon_0 C_3\le1/2$. Then
\eqref{eq:1+eta_kappa_bnd} follows immediately from
\eqref{eq:bnds_for_g'_kappa}. The determinant bounds follow from
\eqref{eq:det_J_of_P}, \eqref{eq:det_J_of_Pinv}, and the two-sided bound on
$\norm{\mathbf g'}$. The matrix-norm bounds follow from
\eqref{eq:J_of_P}, \eqref{eq:J_of_Pinv}, and
\eqref{eq:1+eta_kappa_bnd}.
\end{proof}

	
Without loss of generality, we assume $\epsilon$ is sufficiently small that the boundary and interface tubular neighborhoods are disjoint:
%
\[
  N_{\Gamma_b}(\epsilon)\cap N_{\Gamma_i}(\epsilon)=\emptyset,
\]
and that, for $0<\epsilon\le\epsilon_0$, there exists $h_0>0$ such that
\begin{equation}
  \bigcup_{K\in\mathcal{T}_h^c}K
  \subset N_{\Gamma_b}(\epsilon)\cup N_{\Gamma_i}(\epsilon),
  \qquad
  \bigcup_{K\in\mathcal{T}_h^i}K
  \subset N_{\Gamma_i}(\epsilon),
  \qquad h\le h_0.
  \label{eq:bnd_intf_elem_in_epsNei}
\end{equation}
Thus, every curved element or interface element lies in the tubular neighborhood
of its associated curve. Unless stated otherwise,  $\Gamma$ denotes
either $\Gamma_b$ or $\Gamma_i$, as determined by the element under
consideration.

\subsection{Projected curve segments}
\label{subsec:projected_segments}
For $K\in\mathcal{T}_h^c\cup\mathcal{T}_h^i$, define
\begin{equation}
  \Xi_K:=\bigl\{\xi(X):X\in K\bigr\},
  \qquad
  \xi_m(K):=\min\Xi_K,
  \qquad
  \xi_M(K):=\max\Xi_K,
  \label{eq:Xi_K_tri_xim_xiM}
\end{equation}
and consider its orthogonal projection onto $\Gamma$:
\begin{equation}
  \Gamma_{K,\mathrm{Proj}}
  :=\bigl\{\mathbf g(\xi):
       \xi_m(K)\le\xi\le\xi_M(K)\bigr\}.
  \label{eq:Gamma_{P(K)}}
\end{equation}
The related vertex-based quantities are
\begin{equation}
  \widehat\Xi_K:=\bigl\{\xi(A_j):j=1,2,3\bigr\},
  \qquad
  \widehat\xi_m(K):=\min\widehat\Xi_K,
  \qquad
  \widehat\xi_M(K):=\max\widehat\Xi_K,
  \label{eq:hXi_K_tri_xim_xiM}
\end{equation}
and
\begin{equation}
  \Gamma_{K,mM}
  :=\bigl\{\mathbf g(\xi):
       \widehat\xi_m(K)\le\xi\le\widehat\xi_M(K)\bigr\}.
  \label{eq:Gamma_{mM}}
\end{equation}

\begin{lemma}\label{lem:K_proj_to_Gamma}
Assume that $0<\epsilon\le\epsilon_0$ and
\eqref{eq:bnd_intf_elem_in_epsNei} holds. Then
\[
  \Gamma_{K,\mathrm{Proj}}=\Gamma_{K,mM}
\]
for every $K\in\mathcal{T}_h^c\cup\mathcal{T}_h^i$ with vertices $ A_1$, $A_2$, and $A_3$.
\end{lemma}

\begin{proof}
 Since $\nabla\xi\ne0$ in $N_\Gamma(\epsilon)$, the function $\xi(X)$
has no interior local extremum in $K$. Hence its extrema over the
element $K$ are attained on $\partial K$ because $K$ is a compact set. 

Each edge of $K$ is either a segment of $\Gamma$ or a straight line segment.
On a curve edge, $\xi$ is the curve parameter and is monotone. On a straight
edge, the level sets of $\xi$ are the normal segments
$\{\mathbf g(\xi)+\eta\bfn(\xi):\abs\eta\le\epsilon\}$. A straight edge
meets each such level set at most once unless it overlaps one, in which case
$\xi$ is constant on the overlap. Therefore, the restriction of $\xi$ to each
edge is monotone or constant, and its edgewise extrema occur at the endpoints.
Consequently,
\[
  \xi_m(K)=\min_{1\le j\le3}\xi(A_j)=\widehat\xi_m(K),
  \qquad
  \xi_M(K)=\max_{1\le j\le3}\xi(A_j)=\widehat\xi_M(K).
\]
The conclusion follows from \eqref{eq:Gamma_{P(K)}} and
\eqref{eq:Gamma_{mM}}.
\end{proof}

The next two lemmas show that the length of the projected curve segment is
comparable to the element diameter.
\begin{lemma}\label{lem:Gamma_{K, Proj}_UpperBnd}
Under the assumptions of Lemma~\ref{lem:K_proj_to_Gamma}, there exists a
constant $C_u$, independent of $K$ and $h$, such that
\begin{equation}
  \abs{\Gamma_{K,\mathrm{Proj}}}\le C_u h_K,
  \qquad K\in\mathcal{T}_h^c\cup\mathcal{T}_h^i.
  \label{eq:Gamma_{K, Proj}_UpperBnd}
\end{equation}
\end{lemma}

%
\begin{proof}
By \eqref{eq:grad_xi}, \eqref{eq:bnds_for_g'_kappa}, and
\eqref{eq:1+eta_kappa_bnd},
\begin{equation}
  \norm{\nabla\xi(X)}\le\frac{2}{C_1},
  \qquad X\in N_\Gamma(\epsilon).
  \label{eq:Gamma_{K, Proj}_UpperBnd_2}
\end{equation}
By Lemma~\ref{lem:K_proj_to_Gamma}, there exist vertices $X_1,X_2$ of $K$ such
that 
\[\xi_m(K)=\xi(X_1),\qquad \xi_M(K)=\xi(X_2).\] 
For $h_0$ sufficiently small, the segment $\overline{X_1X_2}$ lies in a fixed tubular neighborhood on which \eqref{eq:Gamma_{K, Proj}_UpperBnd_2} holds. The mean value theorem therefore gives 
	\begin{align}
		\abs{\xi_M(K) - \xi_m(K)}  \leq \sup_{X\in\overline{X_1X_2}} \norm{\nabla \xi(X)} \norm{X_2 - X_1} \leq \frac{2}{C_1} h_K. \label{eq:Gamma_{K, Proj}_UpperBnd_3}
	\end{align}
Therefore,
\[
  \abs{\Gamma_{K,\mathrm{Proj}}}
  =\int_{\xi_m(K)}^{\xi_M(K)}\norm{\mathbf g'(\xi)}\,d\xi
  \le C_2\abs{\xi_M(K)-\xi_m(K)}
  \le\frac{2C_2}{C_1}h_K.
\]
Thus \eqref{eq:Gamma_{K, Proj}_UpperBnd} holds with
$C_u=2C_2/C_1$.
\end{proof}

\begin{lemma}\label{lem:Gamma_{K, Proj}_LowerBnd}
Assume that $0<\epsilon\le\epsilon_0$ and that
\eqref{eq:bnd_intf_elem_in_epsNei} holds for $h_0$ sufficiently small. Then
there exists a constant $C_l>0$, independent of $K$ and $h$, such that
\begin{equation}
  \abs{\Gamma_{K,\mathrm{Proj}}}\ge C_l h_K,
  \qquad K\in\mathcal{T}_h^c\cup\mathcal{T}_h^i.
  \label{eq:Gamma_{K, Proj}_LowerBnd}
\end{equation}
\end{lemma}

\begin{proof} 
Let $X_0\in K$ be the center of an inscribed ball of radius $\rho_K$,
so that $\overline B(X_0,\rho_K)\subset K$. Set
$\xi_0:=\xi(X_0)$ and
\[
  X_+:=X_0+\rho_K\bftau(\xi_0),
  \qquad
  X_-:=X_0-\rho_K\bftau(\xi_0).
\]
Both points belong to $K$. For
$X(t):=tX_++(1-t)X_-$, $0\le t\le1$, the segment
$\{X(t):0\le t\le1\}$ lies in the inscribed ball and hence in $K$.
Moreover, by \eqref{eq:Gamma_{K, Proj}_UpperBnd_2} and
\eqref{eq:ShapeReg},
\[
  \abs{\xi(X(t))-\xi_0}
  \le\frac{2}{C_1}\norm{X(t)-X_0}
  \le\frac{2\rho_K}{C_1}
  \le C h_K.
\]
Since $\bftau$ is uniformly continuous, for $h_0$ sufficiently small,
\[
  \bftau(\xi(X(t)))\cdot\bftau(\xi_0)\ge\frac12,
  \qquad 0\le t\le1.
\]
Using \eqref{eq:grad_xi}, \eqref{eq:bnds_for_g'_kappa}, and
\eqref{eq:1+eta_kappa_bnd}, we obtain
\[
  \nabla\xi(X(t))\cdot\bftau(\xi_0)
  =\frac{\bftau(\xi(X(t)))\cdot\bftau(\xi_0)}
  {\norm{\mathbf g'(\xi(X(t)))}
   \bigl(1+\eta(X(t))\kappa(\xi(X(t)))\bigr)}
  \ge\frac{1}{4C_2}.
\]
Therefore,
\begin{align}
  \xi(X_+)-\xi(X_-)
  &=2\rho_K\int_0^1
    \nabla\xi(X(t))\cdot\bftau(\xi_0)\,dt,
  \label{eq:Gamma_{K, Proj}_LowerBnd_1}\\
  \xi(X_+)-\xi(X_-)
  &\ge\frac{\rho_K}{2C_2}
  \ge\frac{h_K}{2C_{\rm sh}C_2}.
  \label{eq:Gamma_{K, Proj}_LowerBnd_2}
\end{align}
In particular,
\begin{equation}
  \xi_m(K)\le\xi(X_-)<\xi(X_+)\le\xi_M(K).
  \label{eq:Gamma_{K, Proj}_LowerBnd_3}
\end{equation}
Finally,
\begin{align}
  \abs{\Gamma_{K,\mathrm{Proj}}}
  &=\int_{\xi_m(K)}^{\xi_M(K)}\norm{\mathbf g'(\xi)}\,d\xi\ge C_1\bigl(\xi_M(K)-\xi_m(K)\bigr)
  \ge\frac{C_1}{2C_{\rm sh}C_2}h_K.
  \label{eq:Gamma_{K, Proj}_LowerBnd_4}
\end{align}
This proves the result with $C_l=C_1/(2C_{\rm sh}C_2)$.
\end{proof}

\subsection{Fictitious elements and finite overlap}
\label{subsec:fictitious_elements}

For $K\in\mathcal{T}_h^c\cup\mathcal{T}_h^i$, define the fictitious element
\begin{equation*}
  K_F
  :=P_\Gamma\bigl([-h_K,h_K]\times
    [\xi_m(K),\xi_M(K)]\bigr)
\end{equation*}
which is a curved quadrilateral bounded by the two curves
\[
  P_\Gamma(\pm h_K,\xi),
  \qquad \xi_m(K)\le\xi\le\xi_M(K),
\]
and by the two normal segments through $\mathbf g(\xi_m(K))$ and
$\mathbf g(\xi_M(K))$. Since every element in
$\mathcal{T}_h^c\cup\mathcal{T}_h^i$ intersects its associated curve and has
diameter $h_K$, we have $K\subset K_F$. Its image in Frenet coordinates is the rectangle
\[
  \widehat K_F:=R_\Gamma(K_F)
  =[-h_K,h_K]\times[\xi_m(K),\xi_M(K)].
\]
The curve segment $\Gamma_{K,\mathrm{Proj}}$ divides $K_F$ into $K_F^-$ and
$K_F^+$, while its image $\widehat\Gamma_{K,\mathrm{Proj}}$ divides
$\widehat K_F$ into $\widehat K_F^-$ and $\widehat K_F^+$. These sets are
illustrated in Figure~\ref{fig:mapping_xy_etaxi_tri}.
	
For $X\in\Omega$, define
\begin{equation}
  \mathcal F_h(X)
  :=\bigl\{K_F:X\in K_F,
       \ K\in\mathcal{T}_h^c\cup\mathcal{T}_h^i\bigr\}.
  \label{eq:X_FElem}
\end{equation}

\begin{lemma}\label{lem:bnd_for_sum_of_chi_{K_F}(X)}
Assume in addition that $\mathcal{T}_h$ is quasi-uniform. Then there exists a
constant $C$, independent of $h$ and $X$, such that
\begin{equation}
  \sum_{K\in\mathcal{T}_h^c\cup\mathcal{T}_h^i}
  \chi_{K_F}(X)\le C,
  \qquad X\in\Omega.
  \label{eq:bnd_for_sum_of_chi_{K_F}(X)}
\end{equation}
\end{lemma}

\begin{proof}
For $X\in\Omega$, let $K_F\in\mathcal F_h(X)$. Since
$\xi(X)\in[\xi_m(K),\xi_M(K)]$ and $\xi(K)$ is a connected interval, there
exists $Y\in K$ such that $\xi(Y)=\xi(X)$. The element $K$ intersects its
associated curve $\Gamma$, so $\abs{\eta(Y)}\le h_K$. By the definition of
$K_F$, $\abs{\eta(X)}\le h_K$. Since $X$ and $Y$ lie on the same normal segment,
\[
  \norm{X-Y}=\abs{\eta(X)-\eta(Y)}\le2h_K\le2h.
\]
Hence every element associated with a member of $\mathcal F_h(X)$ intersects
$B(X,2h)$ and is contained in $B(X,3h)$.

By shape regularity and quasi-uniformity, there is a constant $c>0$ such that
$\abs K\ge c h^2$ for every $K\in\mathcal{T}_h$. Since the elements have
disjoint interiors,
\[
  c h^2\,\abs{\mathcal F_h(X)}
  \le\sum_{K_F\in\mathcal F_h(X)}\abs K
  \le\abs{B(X,3h)}
  =9\pi h^2.
\]
Here, $\abs{\mathcal F_h(X)}$ denotes the number of elements in the set $\mathcal F_h(X)$. Thus $\abs{\mathcal F_h(X)}\le9\pi/c$, which implies
\eqref{eq:bnd_for_sum_of_chi_{K_F}(X)}.
\end{proof}

%
%
	
	\begin{figure}[!htb]
		\centerline{
			\includegraphics[width=.78\textwidth]{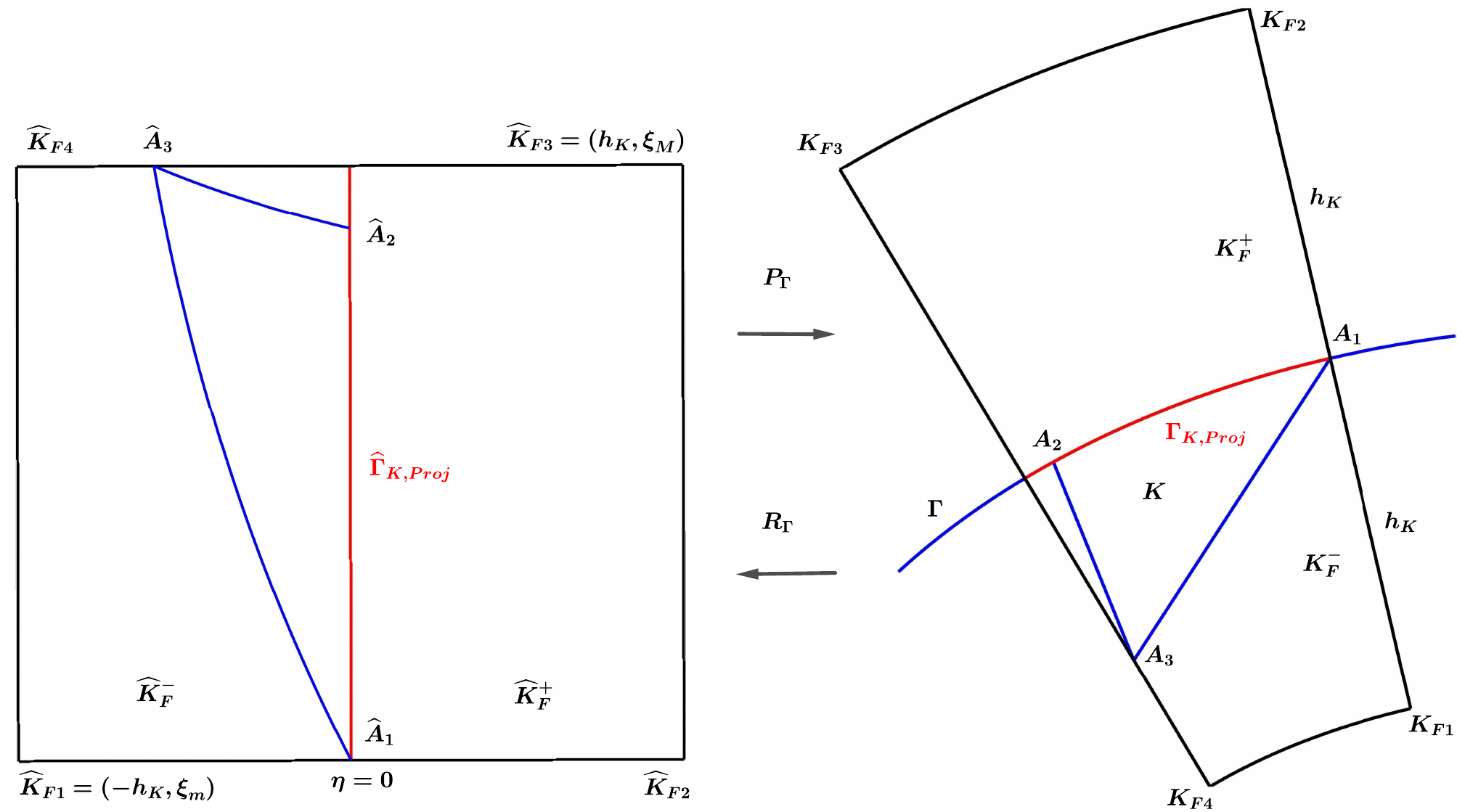}
		}
		\caption{The mappings $P_\Gamma$ and $R_\Gamma$ between
  $K_F=K_F^-\cup K_F^+$ and
  $\widehat K_F=\widehat K_F^-\cup\widehat K_F^+$ for $K$ associated with either a curved boundary or a fitted interface.}
		\centering
		\label{fig:mapping_xy_etaxi_tri}
	\end{figure}

\subsection{Norm equivalence under the Frenet transformation}
\label{subsec:norm_equivalence}

Let $D\subset N_\Gamma(\epsilon)$ be an open subset and set
$\widehat D:=R_\Gamma(D)$. For $v:D\to\mathbb R$, define $\widehat v:=v\circ P_\Gamma$. If $v\in H^1(D)$, then $\widehat v\in H^1(\widehat D)$ and
\begin{equation}
  \nabla v=(DR_\Gamma)^T\widehat\nabla\widehat v,
  \qquad
  \widehat\nabla\widehat v=(DP_\Gamma)^T\nabla v.
  \label{eq:gradient_relations}
\end{equation}

The following lemma gives the equivalence of the $L^2$ norms of $v$ and $\hat v$, and of their gradients.
\begin{lemma}\label{lem:L^2norms_equi_KF_hKF}
For $0<\epsilon\le\epsilon_0$, there exists a constant $C$, independent of
$D$ and $v$, such that
\begin{align}
  C^{-1}\norm{v}_{0,D}
  \le\norm{\widehat v}_{0,\widehat D}
  \le C\norm{v}_{0,D},
  &\hspace{.5in} v\in L^2(D),
  \label{eq:L^2norms_equi_KF_hKF}\\
  C^{-1}\norm{\nabla v}_{0,D}
  \le\norm{{\widehat\nabla\widehat v}}_{0,\widehat D}
  \le C\norm{\nabla v}_{0,D},
  &\hspace{.5in}v\in H^1(D).
  \label{eq:L^2norms_equi_KF_hKF_grad}
\end{align}
\end{lemma}
\begin{proof} The change-of-variables formula and
Lemma~\ref{lem:Frenet_uniform_bounds} give
\[
  \norm{v}_{0,D}^2
  =\int_{\widehat D}\widehat v(\eta,\xi)^2
    \abs{\det DP_\Gamma(\eta,\xi)}\,d\eta\,d\xi
  \simeq\norm{\widehat v}_{0,\widehat D}^2.
\]
Using \eqref{eq:gradient_relations}, the uniform matrix bounds in
Lemma~\ref{lem:Frenet_uniform_bounds}, and the same change of variables yields
\[
  \norm{\nabla v}_{0,D}
  \le C\norm{\widehat\nabla\widehat v}_{0,\widehat D},
  \qquad
  \norm{\widehat\nabla\widehat v}_{0,\widehat D}
  \le C\norm{\nabla v}_{0,D}.
\]
This proves both assertions.
\end{proof}

\begin{remark}\label{rem:GCFE_higher_order_Frenet_norms}
The norm equivalences in Lemma~\ref{lem:L^2norms_equi_KF_hKF} extend to higher
Sobolev orders. If $\Gamma$ is of class $C^{m+2}$, then there exists a constant $C\ge 1$ such that 
\[
  C^{-1}\norm{v}_{j,D}
  \le\norm{\widehat v}_{j,\widehat D}
  \le C\norm{v}_{j,D},
  \qquad 0\le j\le m+1.
\]
The constants are uniform for $D\subset N_\Gamma(\epsilon)$ and fixed
$0<\epsilon\le\epsilon_0$. The result follows from repeated applications of
the chain rule, the uniform bounds on the derivatives of $P_\Gamma$ and
$R_\Gamma$, and the determinant bounds in
Lemma~\ref{lem:Frenet_uniform_bounds}.
\end{remark}

\section{Geometry-conforming finite element spaces } 
\label{sec:GCFE_spaces}

In this section, $\{\mathcal{T}_h\}_{h>0}$ is a family of quasi-uniform, interface-fitted curved meshes in the sense of Section~\ref{sec:prelim}; hence, $\mathcal T_h^i = \emptyset$ and $\mathcal T_h=\mathcal T_h^c\cup\mathcal T_h^r$.
We assume that the boundary and interface curves are of class $C^{m+2}$ and
that $h\leq h_0$ is sufficiently small. For each $K\in\mathcal T_h^c$, let
$\Gamma$ denote the associated segment of either $\Gamma_b$ or $\Gamma_i$. Let $\Omega_K$ be the physical subdomain containing $K$. The index $s_K\in\{+,-\}$ is such that
\[
  K\subset K_F^{s_K}\subset\Omega_K.
\]
For a curved-boundary element $K$, $\Omega_K$ is the domain on the interior side
of $\Gamma_b$, and for an interface-fitted element, it is either $\Omega^-$ or
$\Omega^+$. 
{Throughout this section, constants may depend on
the fixed polynomial degree $m$, the geometry, and mesh regularity, but
are independent of $h$ and the diffusion coefficients. The GC-FE spaces
and projections depend only on geometry and degree, and the choice
$K_F^{s_K}\subset\Omega_K$ confines each local projection to a single
physical subdomain.} 
We begin with the construction of the geometry-conforming finite element space on $K \in 
	\mathcal{T}_h^c$. 

\subsection{Local and global GC-FE spaces}
\label{subsec:GCFE_construction}
Let $p_j$, $0\leq j\leq m$, be the Legendre polynomial of degree $j$. For
$0\leq t\leq m$, let $q_t$ be a polynomial of degree exactly $t$ satisfying
\[
  q_0(x) = p_0(x) = 1, \qquad 
  q_1(0) = 0, \qquad q_i(0) = q_i'(0) = 0, \quad 2 \leq i \leq m.
\]
One particular family $\{q_t\}_{t=0}^m$ is given in
\citet{AdjeridLinMeghaichi2026}; other families satisfying these conditions may
also be used.
Let $K\in\mathcal T_h^c$ have vertices $A_1,A_2,A_3$, and write
$(\eta_i,\xi_i)=R_\Gamma(A_i)$. Define
\begin{align}
  \eta_{h,K}
  &:=\max_{1\leq i\leq3}\abs{\eta_i},
  &
  \xi_K^{\mathtt{mid}}
  &:=\frac{\xi_m(K)+\xi_M(K)}{2},
  &
  \xi_{h,K}
  &:=\frac{\xi_M(K)-\xi_m(K)}{2}.
  \label{eq:xi_01_etah_xih}
\end{align}
The nondegeneracy of $K$ gives $\eta_{h,K}>0$, and
Lemmas~\ref{lem:Gamma_{K, Proj}_UpperBnd} and
\ref{lem:Gamma_{K, Proj}_LowerBnd}, together with the uniform bounds on
$\norm{\mathbf g'}$, imply $\xi_{h,K}\simeq h_K$.

Set
\begin{align}
  \widehat q_t(\eta)
  &:=q_t\!\left(\frac{\eta}{\eta_{h,K}}\right),
  &
  \widehat p_j(\xi)
  &:=p_j\!\left(\frac{\xi-\xi_K^{\mathtt{mid}}}{\xi_{h,K}}\right),
  \nonumber\\
  \widehat\phi_{k(t,j)}(\eta,\xi)
  &:=\widehat q_t(\eta)\widehat p_j(\xi),
  &
  k(t,j)
  &:=\frac{(t+j)(t+j+1)}{2}+t+1,
  \label{eq:p_q_R_Polynomials}
\end{align}
for $0\leq t\leq m$ and $0\leq j\leq m-t$. The index map $k(t,j)$ is the
Cantor pairing function restricted to pairs with $t+j\leq m$. Its inverse is
\begin{equation}
  d(k):=\left\lfloor\frac{\sqrt{8k-7}-1}{2}\right\rfloor,
  \qquad
  t(k):=k-1-\frac{d(k)(d(k)+1)}{2},
  \qquad
  j(k):=d(k)-t(k).
  \label{eq:Cantor_pairing_fun_inv}
\end{equation}
Because $q_t$ and ${p_j}$ have exact degrees $t$ and ${j}$, respectively, the
functions in \eqref{eq:p_q_R_Polynomials} form a basis of the total-degree
polynomial space:
\begin{equation}
  \mathbb P_m(\widehat K_F^{s_K})
  =\operatorname{span}\bigl\{\widehat\phi_k:
    1\leq k\leq (m+1)(m+2)/2\bigr\}.
  \label{eq:GCFE_reference_polynomial_span}
\end{equation}
Moreover,
\begin{equation}
  \widehat\phi_k(0,\xi)=0
  \qquad\text{whenever }t(k)\geq1.
  \label{eq:GCFE_trace_basis_property}
\end{equation}
Thus, only the modes with $t(k)=0$ contribute to the trace on the physical
curve $\eta=0$.

We then define the GC-FE space on the physical-side fictitious element by
\begin{align}
  S_h^m(K_F^{s_K})
  &:=\bigl\{\widehat v\circ R_\Gamma:
       \widehat v\in\mathbb P_m(\widehat K_F^{s_K})\bigr\}
  \nonumber\\
  &=\operatorname{span}\bigl\{\phi_k :=\widehat\phi_k\circ R_\Gamma:
       1\leq k\leq(m+1)(m+2)/2\bigr\}.
  \label{eq:GCFE_on_K_F^-}
\end{align}
The local space on the curved element is the restriction
\begin{equation}
  S_h^m(K) := \bigl\{v|_K:v\in S_h^m(K_F^{s_K})\bigr\} = \operatorname{span}\bigl\{ \phi_k|_K \bigr\}, \qquad K\in\mathcal T_h^c.
  \label{eq:GCFE_on_K}
\end{equation}
Since a polynomial that vanishes on the set $R_\Gamma(K)$ vanishes
identically, the restriction is injective and
\[
  \dim S_h^m(K)=\frac{(m+1)(m+2)}{2} =:d_m.
\]
For $K\in\mathcal T_h^r$, set $S_h^m(K):=\mathbb P_m(K)$ as usual. The global
GC-FE space is then defined as 
\begin{equation}
  S_h^m(\mathcal{T}_h)
  :=\bigl\{v\in L^2(\Omega):v|_K\in S_h^m(K)
        \text{ for every }K\in\mathcal T_h\bigr\}.
  \label{eq:GCFE_space_Omega}
\end{equation}

\begin{remark}[Relation to isogeometric approximation]
\label{rem:GCFE_IGA}
The scalar isogeometric space associated with a regular geometry map
$F:\widehat\Omega\to\Omega$ has the form
\[
  V_h^{\mathrm{IGA}}
  =\{\widehat v\circ F^{-1}:\widehat v\in\widehat V_h\},
\]
where $\widehat V_h$ is a spline or NURBS space; see
\citet[Section~3.1]{HughesSangalli2017}. Equations
\eqref{eq:GCFE_on_K_F^-}--\eqref{eq:GCFE_on_K} have the same mapped-space
structure, with $F=P_\Gamma$ and a local total-degree polynomial space.
In GC-FE, however, the geometry map is prescribed by the curve and its
normal field and need not belong componentwise to that polynomial space.
\end{remark}

	
\subsection{GC-FE basis reconstruction}
\label{subsec:GC-FE Basis Reconstruction}

The tensor-product-type basis
\(\{\widehat\phi_k:1\leq k\leq d_m\}\) defined in
\eqref{eq:p_q_R_Polynomials} is commonly used in finite element computation
and analysis, but its associated local mass matrix can become increasingly
ill-conditioned as the polynomial degree increases. Such ill-conditioning
indicates that the basis functions are nearly linearly dependent and can amplify roundoff errors and degrade numerical stability. To improve the conditioning, we adapt the basis-reconstruction procedures developed for GC-IFE spaces in
\citet{AdjeridLinMeghaichi2026} and \citet{2026LinLinZhang}.

To describe the basis reconstructions, let
\[
\boldsymbol\phi_K
:=
(\phi_1^K,\ldots,\phi_{d_m}^K)^T,
\qquad
\phi_j^K
:=
(\widehat\phi_j\circ R_\Gamma)|_K,
\]
be the original basis of the local GC-FE space \(S_h^m(K)\) defined in
\eqref{eq:GCFE_on_K}, where \(K\in\mathcal T_h^c\).

Choose a quadrature rule on \(K\) with nodes and positive weights
\(\{(X_r,w_r)\}_{r=1}^{N_q}\), and define
\[
(\mathsf V_K)_{rj}:=\phi_j^K(X_r),
\qquad
\mathsf W_K:=\operatorname{diag}(w_1,\ldots,w_{N_q}),
\qquad
\mathsf A_K:=\mathsf W_K^{1/2}\mathsf V_K.
\]
Assume that \(N_q\geq d_m\) and that \(\mathsf A_K\) has full column rank.
The quadrature-based local mass matrix is then
\[
\mathsf M_{q,K}
:=
\mathsf V_K^T\mathsf W_K\mathsf V_K
=
\mathsf A_K^T\mathsf A_K,
\]
which is symmetric positive definite.

For any nonsingular matrix
\(\mathsf Q\in\mathbb R^{d_m\times d_m}\), the transformed functions
$\widetilde{\boldsymbol\phi}_K
:=
\mathsf Q^T\boldsymbol\phi_K$ form another basis of \(S_h^m(K)\), whose quadrature-based mass matrix is
\begin{equation}
\widetilde{\mathsf M}_{q,K}
=
\mathsf Q^T\mathsf M_{q,K}\mathsf Q.
\label{eq:GCFE_reconstructed_mass}
\end{equation}

In the first reconstruction approach (RA1), let
\(
\mathsf M_{q,K}
=
\mathsf V_1\boldsymbol\Lambda\mathsf V_1^T
\)
be its spectral decomposition and set
\(
\mathsf Q_1
:=
\mathsf V_1\boldsymbol\Lambda^{-1/2},
\quad
\widetilde{\boldsymbol\phi}_K^{(1)}
:=
\mathsf Q_1^T\boldsymbol\phi_K.
\)
It follows from \eqref{eq:GCFE_reconstructed_mass} that
\(
\widetilde{\mathsf M}_{q,K}^{(1)}
=
\mathsf Q_1^T\mathsf M_{q,K}\mathsf Q_1
=
\mathsf I.
\)
Thus, RA1 produces a basis orthonormal with respect to the quadrature inner
product.

In the second reconstruction approach (RA2), we compute the reduced singular
value decomposition
\(
\mathsf A_K
=
\mathsf U_2\boldsymbol\Sigma\mathsf V_2^T
\)
and set
\(
\mathsf Q_2
:=
\mathsf V_2\boldsymbol\Sigma^{-1},
\quad
\widetilde{\boldsymbol\phi}_K^{(2)}
:=
\mathsf Q_2^T\boldsymbol\phi_K.
\)
Again, by \eqref{eq:GCFE_reconstructed_mass}, we have
\(
\widetilde{\mathsf M}_{q,K}^{(2)}
=
\mathsf Q_2^T\mathsf M_{q,K}\mathsf Q_2
=
\mathsf I.
\)

Thus, in exact arithmetic, both RA1 and RA2 produce bases orthonormal with
respect to the quadrature inner product and hence identity quadrature-based
mass matrices, although the reconstructed bases may differ by an orthogonal
transformation. In finite precision, however, because
\[
  \operatorname{cond}_2(\mathsf M_{q,K})
  =\operatorname{cond}_2(\mathsf A_K)^2,
\]
computing the SVD used by RA2 is generally less susceptible to roundoff error. The numerical comparison in Subsection~\ref{sec:matrix_conditioning} also corroborates the greater
robustness of RA2 for high polynomial degrees. The
reconstruction changes only the basis, not the local GC-FE space, and therefore
does not affect the approximation or stability analysis below. Accordingly, we recommend the basis produced by RA2 for GC-FE computations.

\subsection{Projection operators and approximation estimates}
\label{subsec:GCFE_approximation}

\leavevmode
To establish approximation estimates for the GC-FE spaces, we construct a projection operator on a rectangle in Frenet coordinates and map the resulting estimates to physical coordinates. The rectangle is selected so that its physical image lies entirely within a single subdomain.

For $K\in\mathcal T_h^c$, let
\[
  \widehat\Pi_{\widehat K_F^{s_K}}:
  L^2(\widehat K_F^{s_K})\longrightarrow
  \mathbb P_m(\widehat K_F^{s_K})
\]
be the standard $L^2$ projection, which is characterized by
\begin{equation}
  \bigl(\widehat\Pi_{\widehat K_F^{s_K}}\widehat u,
        \widehat v\bigr)_{0,\widehat K_F^{s_K}}
  =\bigl(\widehat u,\widehat v\bigr)_{0,\widehat K_F^{s_K}}
  \qquad
  \forall\widehat v\in\mathbb P_m(\widehat K_F^{s_K}).
  \label{eq:L^2Proj_on_hK^-}
\end{equation}
For
$u\in L^2(K_F^{s_K})$, set $\widehat u:=u\circ P_\Gamma$ and define the
Frenet {\it pullback} projection
\begin{equation}
  \Pi_{K_F^{s_K}}u
  :=\bigl(\widehat\Pi_{\widehat K_F^{s_K}}\widehat u\bigr)
       \circ R_\Gamma.
  \label{eq:Map_Pi_on_K_F^-}
\end{equation}

The corresponding local operator is defined as
\begin{equation}
  \Pi_K u:=\bigl(\Pi_{K_F^{s_K}}(u|_{K_F^{s_K}})\bigr)|_K,
  \qquad K\in\mathcal T_h^c.
  \label{eq:Local_L2_proj_on_CurvedK}
\end{equation}
For $K\in\mathcal T_h^r$, $\Pi_K$ is the usual $L^2(K)$ projection onto
$\mathbb P_m(K)$. Finally, define $\Pi_h:L^2(\Omega)\to S_h^m(\mathcal{T}_h)$
elementwise by
\begin{equation}
  (\Pi_hu)|_K:=\Pi_K u,
  \qquad K\in\mathcal T_h.
  \label{eq:Local_L2_proj_on_Omega}
\end{equation}

	We first estimate the local projection error on a curved element  $K \in \mathcal{T}_h^c$. 
\begin{lemma}\label{lem:L2Proj_error_bnds_K}
There exists a constant $C$, independent of $K$
and $h$, such that for every $K\in\mathcal T_h^c$ and every
$u\in H^{m+1}(K_F^{s_K})$,
\begin{equation}
  \abs{\Pi_K u-u}_{k,K}
  \leq C h_K^{m+1-k}\norm{u}_{m+1,K_F^{s_K}},
  \qquad 0\leq k\leq m+1.
  \label{eq:L2Proj_error_bnds_K}
\end{equation}
\end{lemma}

\begin{proof}
Note that $\widehat K_F^{s_K}$ is a rectangle and its side in the
$\eta$-direction has length $h_K$. Lemmas~\ref{lem:Gamma_{K, Proj}_UpperBnd}
and~\ref{lem:Gamma_{K, Proj}_LowerBnd}, together with the uniform bounds on
$\norm{\mathbf g'}$, show that its side in the $\xi$-direction has a length
comparable to $h_K$. Thus these rectangles form a uniformly shape-regular
family after scaling by $h_K$.

The standard approximation estimate for the $L^2$ projection onto
$\mathbb P_m(\widehat K_F^{s_K})$ gives
\begin{equation}
  \abs{\widehat\Pi_{\widehat K_F^{s_K}}\widehat u-\widehat u}
       _{k,\widehat K_F^{s_K}}
  \le C h_K^{m+1-k}\norm{\widehat u}_{m+1,\widehat K_F^{s_K}},
  \qquad 0\le k\le m+1.
  \label{eq:L2Proj_error_bnds}
\end{equation}
The higher-order norm equivalence in
Remark~\ref{rem:GCFE_higher_order_Frenet_norms} transfers this estimate to
$K_F^{s_K}$. Restricting the result to $K\subset K_F^{s_K}$ proves
\eqref{eq:L2Proj_error_bnds_K}.
\end{proof}

Lemma \ref{lem:L2Proj_error_bnds_K} can be extended to functions of lower regularity as
described in the following corollary. 
\begin{corollary}
\label{cor:GCFE_DG_projection_estimate}
Let $K\in\mathcal T_h^c$, and choose $s_K\in\{+,-\}$ so that
\begin{equation}
  K\subset K_F^{s_K}\subset\Omega_K.
  \label{eq:GCFE_fictitious_side}
\end{equation}
Let $2\leq r\leq m+1$ and 
$0\leq j\leq\min\{r,2\}$. Then there exists a constant $C$ independent of $K$ and $h$ such that 
\begin{equation}
  \abs{\Pi_Kv - v}_{j,K}
  \leq C h_K^{r-j}\norm{v}_{r,K_F^{s_K}},
  \qquad v\in H^r(K_F^{s_K}).
  \label{eq:GCFE_DG_projection_estimate}
\end{equation}
For $K\in\mathcal T_h^r$, the same
estimate holds with $K_F^{s_K}$ replaced by $K$.
\end{corollary}
\begin{proof} The result follows from the proof of
Lemma~\ref{lem:L2Proj_error_bnds_K}, applied with regularity index \(r\).
The inclusion \(K_F^{s_K}\subset\Omega_K\) ensures that the relevant Sobolev
norm is taken over a single smooth physical subdomain. For
\(K\in\mathcal T_h^r\), the result reduces to the standard polynomial
projection estimate.
\end{proof}	
For a piecewise $H^k$ function, set
\[
  \abs{w}_{k,\mathcal T_h}^2
  :=\sum_{K\in\mathcal T_h}\abs{w}_{k,K}^2.
\]
The following theorem gives the global approximation property of $S_h^m(\mathcal{T}_h)$.
\begin{theorem}\label{th:L2Proj_error_bnds_Omega}
There exists a constant $C$, independent of $h$, such that for $0\leq k\leq m+1$
\begin{equation}
  \abs{\Pi_hu-u}_{k,\mathcal T_h}
  \leq C h^{m+1-k}\norm{u}_{m+1,\Omega},
  \qquad u\in H^{m+1}(\Omega).
  \label{eq:L2Proj_error_bnds_Omega}
\end{equation}
\end{theorem}

\begin{proof} Lemma~\ref{lem:L2Proj_error_bnds_K} and the standard polynomial
projection estimate on $K\in\mathcal T_h^r$ yield
\begin{align*}
  \abs{\Pi_hu-u}_{k,\mathcal T_h}^2
  &\leq C h^{2(m+1-k)}
  \left(
    \sum_{K\in\mathcal T_h^c}
      \norm{u}_{m+1,K_F^{s_K}}^2
    +\sum_{K\in\mathcal T_h^r}\norm{u}_{m+1,K}^2
  \right).
\end{align*}
Since $K_F^{s_K}\subset K_F$, the finite-overlap estimate
of Lemma~\ref{lem:bnd_for_sum_of_chi_{K_F}(X)} gives
\begin{align*}
  \sum_{K\in\mathcal T_h^c}\norm{u}_{m+1,K_F^{s_K}}^2
  &\leq
  \sum_{\abs{\alpha}\leq m+1}
  \int_\Omega \abs{D^\alpha u(X)}^2
  \left(\sum_{K\in\mathcal T_h^c}\chi_{K_F}(X)\right)dX
  \leq C\norm{u}_{m+1,\Omega}^2.
\end{align*}
Combining the last two estimates and taking square roots proves
\eqref{eq:L2Proj_error_bnds_Omega}.
\end{proof}

\begin{remark}
The same proof applies to functions that are only piecewise smooth across a
fitted interface. More precisely, if
$u|_{\Omega^\pm}\in H^{m+1}(\Omega^\pm)$, then the estimate
\eqref{eq:L2Proj_error_bnds_Omega} can be replaced by
\[
  \abs{\Pi_hu-u}_{k,\mathcal T_h}\le C h^{m+1-k}
  \left(\norm{u}_{m+1,\Omega^-}^2+
        \norm{u}_{m+1,\Omega^+}^2\right)^{1/2}.
\]
{For piecewise constant $\beta|_{\Omega^\pm}=\beta^\pm>0$,
multiplying the squared local estimates by $\beta_K:=\beta|_K$ and
summing within each subdomain yields, for $0\leq k\leq m+1$,
\[
  \sum_{K\in\mathcal T_h}\beta_K\abs{\Pi_hu-u}_{k,K}^2
  \leq C h^{2(m+1-k)}
       \sum_{s\in\{-,+\}}\beta^s\norm{u}_{m+1,\Omega^s}^2.
\]
The constant $C$ is independent of $\beta^\pm$ and their ratio, since
the projection patches stay within their respective physical subdomains
and the overlap bound is purely geometric.}
\end{remark}

	\subsection{Inverse estimates}
\label{subsec:GCFE_inverse}

Although GC-FE functions are generally nonpolynomial in Cartesian
coordinates, their pullbacks belong to a polynomial space of prescribed
degree. This polynomial structure of their pullbacks further ensures that
GC-FE functions satisfy the inverse estimates stated below.

	\begin{theorem} 
\label{th:GCFE_inverse}
		There exists a constant $C$, independent of $K$ and $h$, such that
		\begin{align}
			\abs{v_h}_{j,K}
			\leq C h_K^{\ell-j}\abs{v_h}_{\ell,K},
			\qquad
			0\leq \ell\leq j\leq 2,
			\qquad
			v_h \in S_h^m(K),~~K\in\mathcal T_h .
			\label{eq:GCFE_inverse_estimates}
		\end{align}
		The constant may depend on the fixed polynomial degree $m$ and the mesh and geometric regularity constants.
	\end{theorem}
	
	\begin{proof}
The result is standard for $K\in\mathcal T_h^r$. Let
$K\in\mathcal T_h^c$, set $\widehat K:=R_\Gamma(K)$, and write
$v_h=\widehat v_h\circ R_\Gamma$ with
$\widehat v_h\in\mathbb P_m$. The uniform bounds for $P_\Gamma$ and
$R_\Gamma$, together with the shape regularity of $K$, imply the existence of
constants $c,C>0$ and points $\widehat X_K$ such that
\[
  B(\widehat X_K,ch_K)\subset\widehat K
  \subset B(\widehat X_K,Ch_K).
\]
After translation and scaling by $h_K$, finite-dimensional norm equivalence
on $\mathbb P_m$ and the usual polynomial inverse estimate therefore give
\begin{equation}
  \abs{\widehat v_h}_{j,\widehat K}
  \le C h_K^{\ell-j}\abs{\widehat v_h}_{\ell,\widehat K},
  \qquad 0\le\ell\le j\le2.
  \label{eq:GCFE_inverse_frenet}
\end{equation}
The first-derivative chain rule and Lemma~\ref{lem:L^2norms_equi_KF_hKF}
yield
$\abs{v_h}_{1,K}\le Ch_K^{-1}\norm{v_h}_{0,K}$. For second derivatives,
the chain rule gives
\[
  \abs{v_h}_{2,K}
  \le C\bigl(\abs{\widehat v_h}_{2,\widehat K}
              +\abs{\widehat v_h}_{1,\widehat K}\bigr)
  \le C h_K^{-1}\abs{v_h}_{1,K},
\]
where $h_K\le h_0$ and the $C^3$ regularity provides uniform bounds for the
second derivatives of $R_\Gamma$. The case $(\ell,j)=(0,2)$ follows by
combining the last two estimates; the remaining cases follow immediately.
\end{proof}


\subsection{Trace estimates}
\label{subsec:GCFE_trace}
For a straight element $K\in\mathcal T_h^r$, the standard scaled trace
inequality is
\begin{equation}
  \norm{v}_{0,\partial K}^2
  \leq C\left(
       h_K^{-1}\norm{v}_{0,K}^2+h_K\abs{v}_{1,K}^2\right),
  \qquad v\in H^1(K).
  \label{eq:TraceIneq_on_Thn}
\end{equation}
We now establish the corresponding result for curved elements.

Let $K\in\mathcal T_h^c$ have curved edge $e_K=\widearc{A_1A_2}$ with endpoints $A_1$ and
$A_2$, and let $A_3$ be its third vertex. Let
$\xi_i:=\xi(A_i)$, $i=1,2$, and let the parameterization of the curved edge be
\[
  \breve{\mathbf g}(t)
  :=\mathbf g\bigl(\xi_1+t(\xi_2-\xi_1)\bigr),
  \qquad 0\leq t\leq1.
\]
Thus, $\breve{\mathbf g}(0) = \mathbf g(\xi_1) = A_1, ~~\breve{\mathbf g}(1) = \mathbf g(\xi_2) = A_2$. 
Since $|\xi_2-\xi_1|\simeq h_K$, the uniform bound on $\mathbf g''$ gives 
\begin{equation}
  \norm{\breve{\mathbf g}''(t)}\leq Ch_K^2,
  \qquad 0\leq t\leq1.
  \label{eq:bnd_breveg''}
\end{equation}
Let $\overline K:=\triangle A_1A_2A_3$ {be the straight triangle with vertices $A_1,A_2,A_3$, and let $\widehat{\overline K}$ be the reference triangle with vertices} $\hat{A}_1 = (0, 0), ~\hat{A}_2=(1, 0)$, and $\hat{A}_3 = (0, 1)$. Define the following blending-function mapping: $F_K: \widehat{\overline{K}} \rightarrow K$ such that
	\begin{align*}
		X = F_K({\hat X}) = \overline{F}_K({\hat X}) + \widehat\phi_K({\hat X})
	\end{align*}
	where $\overline{F}_K: \widehat{\overline{K}}  \rightarrow \overline{K}$ is the standard affine mapping on a triangle such that 
	\begin{align*}
		\overline{F}_K({\hat X}) = B_K {\hat X} + A_1, ~~B_K = [A_2 - A_1, A_3 - A_1], ~~\forall {\hat X} \in\widehat{\overline{K}},
	\end{align*}
	and $\widehat\phi_K({\hat X})$ is a Gordon--Hall-type blending extension function \citep{GordonHall1973b,gordon1973construction}:
	\begin{eqnarray*}
		\widehat\phi_K({\hat X}) = \widehat\phi_K({\hat x}, {\hat y}) = \frac{1 - {\hat x} - {\hat y}}{1 - {\hat x}} \Big(\breve{\mathbf g}({\hat x}) - \big(A_1 + {\hat x}(A_2 - A_1)\big)\Big), \qquad\forall {\hat X} = (\hat{x}, \hat{y}) \in \widehat{\overline{K}}.
	\end{eqnarray*}
	By direct calculation, we have
	\begin{align}
		DF_K(\hat{X}) = B_K + D\widehat\phi_K(\hat{X}) = B_K \left( I + B_K^{-1} D\widehat\phi_K(\hat{X}) \right). \label{eq:FK}
	\end{align}
	Since $\mathcal{T}_h$ is shape-regular, there exists a constant $C$ such that
	\begin{align}
		\|B_K\| \leq C h_K, \quad \|B_K^{-1}\| \leq C h_K^{-1}, \quad |\det B_K| \approx h_K^2. \label{eq:bnds_BK}
	\end{align}

	The following lemma establishes a few essential estimates about the map $F_K$. 
	\begin{lemma}\label{lem:curved_element_map_scaling}
		There exists $h_0>0$ such that for every $K \in \mathcal{T}_h^c, ~~h \leq h_0$, $DF_K(\hat{X})$ is nonsingular, and the following hold for a certain constant $C$:
		\begin{align}
			\begin{split}\label{eq:bnds_for_DF_K}
				&\abs{\det DF_K(\hat{X})} \leq Ch_K^2, \quad \abs{\det (DF_K(\hat{X}))^{-1}} \leq Ch_K^{-2}, \\
				&\norm{DF_K(\hat{X})} \leq Ch_K, \quad \norm{(DF_K(\hat{X}))^{-1}} \leq Ch_K^{-1}. 
			\end{split}
		\end{align}
	\end{lemma}
\begin{proof} 
By \eqref{eq:bnd_breveg''}, we can show that there exists a constant $C$ such that
	\begin{align*}
		\norm{\widehat\phi_K(\hat{X})} \leq Ch_K^2, \quad \norm{D\widehat\phi_K(\hat{X})} \leq Ch_K^2.
	\end{align*}
	Then, by \eqref{eq:bnds_BK}, $\|B_K^{-1} D\widehat\phi_K(\hat{X})\| \leq \|B_K^{-1}\| \|D\widehat\phi_K(\hat{X})\| \leq C h_K^{-1} \cdot h_K^2 = C h_K$.
	Therefore, by \eqref{eq:FK}, there exists $h_0>0$ such that, whenever $h_K\leq h\leq h_0$, $DF_K(\hat{X})$ is nonsingular, and the estimates in \eqref{eq:bnds_for_DF_K} follow from direct calculations. 
\end{proof}
	
	Consequently, for every edge $\widehat e$ of $\widehat{\overline K}$ and every
nonnegative measurable function $f$ on its image $e=F_K(\widehat e)$, we have
\begin{equation}
  \int_e f\,ds
  \leq Ch_K\int_{\widehat e}(f\circ F_K)\,d\widehat s.
  \label{eq:LineInterals_estimate}
\end{equation}

\begin{theorem}\label{th:GCFE_general_trace}
There exist $h_0>0$ and a constant $C$, independent of $K$ and $h$, such
that
\begin{equation}
  \norm{v}_{0,\partial K}^2
  \leq C\left(
       h_K^{-1}\norm{v}_{0,K}^2+h_K\abs{v}_{1,K}^2\right),
  \qquad v\in H^1(K),
  \quad K\in\mathcal T_h,
  \quad h\leq h_0.
  \label{eq:TraceIneq}
\end{equation}
\end{theorem}

\begin{proof} For $K\in\mathcal T_h^r$, this is \eqref{eq:TraceIneq_on_Thn}.  Let $K\in\mathcal T_h^c$ and set
$\widehat v:=v\circ F_K$. Summing \eqref{eq:LineInterals_estimate} over the
three edges and applying the trace theorem on the fixed reference triangle
gives
\[
  \norm{v}_{0,\partial K}^2
  \leq Ch_K\norm{\widehat v}_{0,\partial \widehat{\overline K}}^2
  \leq Ch_K\left(
       \norm{\widehat v}_{0, \widehat{\overline K}}^2
       +\abs{\widehat v}_{1, \widehat{\overline K}}^2\right).
\]
By Lemma~\ref{lem:curved_element_map_scaling} and a change of variables,
\[
  \norm{\widehat v}_{0, \widehat{\overline K}}^2
  \leq Ch_K^{-2}\norm{v}_{0,K}^2,
  \qquad
  \abs{\widehat v}_{1, \widehat{\overline K}}^2
  \leq C\abs{v}_{1,K}^2.
\]
Substituting into the inequality above proves \eqref{eq:TraceIneq}.
\end{proof}

Applying Theorem~\ref{th:GCFE_general_trace} componentwise to the gradient
and then using the inverse estimates for GC-FE functions yields the following
gradient and normal-flux trace estimates.

\begin{corollary} 
\label{cor:GCFE_gradient_trace}
There exists a constant $C$, independent of $K$ and $h$, such that
\begin{equation}
  \norm{\nabla v}_{0,\partial K}^2
  \leq C\left(
       h_K^{-1}\abs{v}_{1,K}^2+h_K\abs{v}_{2,K}^2\right),
  \qquad v\in H^2(K),
  \quad K\in\mathcal T_h.
  \label{eq:GCFE_gradient_trace}
\end{equation}
Moreover, for every $v_h\in S_h^m(K)$,
\begin{equation}
  \norm{\nabla v_h}_{0,\partial K}^2
  \leq Ch_K^{-1}\abs{v_h}_{1,K}^2,
  \qquad K\in\mathcal T_h,
  \label{eq:GCFE_discrete_gradient_trace}
\end{equation}
and consequently
\begin{equation}
  \norm{\nabla v_h\cdot\mathbf n_K}_{0,\partial K}^2
  \leq Ch_K^{-1}\abs{v_h}_{1,K}^2.
  \label{eq:GCFE_discrete_flux_trace}
\end{equation}
\end{corollary}	

\begin{proof}
Applying Theorem~\ref{th:GCFE_general_trace} to each component of
\(\nabla v\) and summing the resulting inequalities gives
\eqref{eq:GCFE_gradient_trace}. For \(v_h\in S_h^m(K)\), the inverse estimate
$\abs{v_h}_{2,K}\leq Ch_K^{-1}\abs{v_h}_{1,K}$
then yields \eqref{eq:GCFE_discrete_gradient_trace}. Finally,
\eqref{eq:GCFE_discrete_flux_trace} follows from $\abs{\nabla v_h\cdot\mathbf n_K}\leq\abs{\nabla v_h}$ on $\partial K$.
\end{proof}

{For a constant coefficient $\beta|_K=\beta_K>0$,
\eqref{eq:GCFE_discrete_flux_trace} also gives
\[
  \frac{h_K}{\beta_K}
  \norm{\beta_K\nabla v_h\cdot\mathbf n_K}_{0,\partial K}^2
  \leq C\beta_K\abs{v_h}_{1,K}^2,
  \qquad v_h\in S_h^m(K),
\]
with the same constant $C$, independent of $\beta_K$. Together with
harmonic edge weights, this local flux bound supports the coefficient-robust
stability analysis in Section~\ref{sec:gcfe-sipdg-analysis}.}

\section{A coefficient-robust SIPDG method with GC-FE spaces}
\label{sec:gcfe-sipdg-analysis}

We develop an SIPDG
method for \eqref{eq:interface-pde}--\eqref{eq:jumps} on the
interface-fitted curved meshes. We adopt harmonic flux averages 
and harmonic penalty weights \citep{2003Dryja,2006BurmanZunino,2009ErnStephansenZunino,2011CaiYeZhang,2017CaiHeZhang2}  
with the GC-FE approximation and trace
estimates of Section~\ref{sec:GCFE_spaces}. This yields an optimal-order
energy error estimate whose constant is independent of the coefficient
jump. The analysis is carried out on the exact physical elements and
edges, so no geometric consistency error is introduced.

Throughout this section, $m\geq1$ is fixed,
$\mathcal T_h^i=\emptyset$, and
$\mathcal T_h=\mathcal T_h^c\cup\mathcal T_h^r$ satisfies the geometric
and quasi-uniformity assumptions of Section~\ref{sec:GCFE_spaces}.
In particular, $\Gamma_b$ and $\Gamma_i$ are of class $C^{m+2}$ and
$h\leq h_0$ is sufficiently small.

\subsection{Harmonic weights and the discrete formulation}
\label{subsec:sipdg-formulation}

For simplicity, we consider a piecewise constant diffusion coefficient.
Let $\beta^-$ and $\beta^+$ be arbitrary positive constants and set
\begin{equation}
 \beta|_{\Omega^\pm}=\beta^\pm,
 \qquad
 \beta_K:=\beta|_K,
 \qquad
 \beta_*:=\min\{\beta^-,\beta^+\}.
 \label{eq:sipdg-piecewise-constant-coefficient}
\end{equation}
Since the mesh is interface-fitted, $\beta_K$ is constant on every
element $K\in\mathcal T_h$.  For $r\geq0$, let
\begin{align*}
 \mathcal H^r(\Omega)
 &:=\{v\in L^2(\Omega):v|_{\Omega^\pm}\in H^r(\Omega^\pm)\},\\
 H^r(\mathcal T_h)
 &:=\{v\in L^2(\Omega):v|_K\in H^r(K)\ \forall K\in\mathcal T_h\}.
\end{align*}
The interface-fitted property gives
$\mathcal H^r(\Omega)\subset H^r(\mathcal T_h)$. We use the
element and edge $L^2$ inner products $(\cdot,\cdot)_K$ and
$\langle\cdot,\cdot\rangle_e$, respectively.

Let $\mathcal E_h$ be the set of mesh edges, with
\[
 \mathcal E_h^b:=\{e\in\mathcal E_h:e\subset\Gamma_b\},
 \qquad
 \mathcal E_h^\circ:=\mathcal E_h\setminus\mathcal E_h^b,
 \qquad
 \mathcal E_h^\Gamma:=\{e\in\mathcal E_h^\circ:e\subset\Gamma_i\}.
\]
For $e=\partial K_1\cap\partial K_2\in\mathcal E_h^\circ$, fix the normal
$\bfn_e$ pointing from $K_1$ to $K_2$. Let $v_j$ be the trace of
$v|_{K_j}$ on $e$ and set $\beta_j=\beta_{K_j}$. These indices
refer to the orientation of the edge and need not agree with the
physical labels $+$ and $-$. Define
\begin{equation}
 \jump{v}_e:=v_1-v_2,
 \qquad
 \omega_{1,e}:=\frac{\beta_2}{\beta_1+\beta_2},
 \qquad
 \omega_{2,e}:=\frac{\beta_1}{\beta_1+\beta_2},
 \qquad
 \beta_{e,H}:=\frac{2\beta_1\beta_2}{\beta_1+\beta_2}.
 \label{eq:sipdg-harmonic-weights}
\end{equation}
The weighted average and its complementary average are defined by
\[
 \aver{v}_{\omega,e}:=\omega_{1,e}v_1+\omega_{2,e}v_2,
 \qquad
 \aver{v}_{\omega^\ast,e}:=\omega_{2,e}v_1+\omega_{1,e}v_2.
\]
Thus, on an interior edge we have
\begin{equation}
 \jump{wv}_e
 =\aver{w}_{\omega,e}\jump{v}_e
  +\jump{w}_e\aver{v}_{\omega^\ast,e}.
 \label{eq:sipdg-weighted-product-rule}
\end{equation}
On a boundary edge $e\subset\partial K\cap\Gamma_b$, choose
$\bfn_e$ outward and set $\jump{v}_e=v|_e$,
$\aver{v}_{\omega,e}=v|_e$, and $\beta_{e,H}=\beta_K$.
Normals on curved edges are understood pointwise. Let $h_e=|e|$;
the mesh and map estimates in Section~\ref{subsec:projected_segments} give
$h_e\simeq h_K$ for $e\subset\partial K$.

The identities and inequalities underlying robustness are
\begin{align}
 &\omega_{1,e}+\omega_{2,e}=1,
 \qquad
 \omega_{1,e}\beta_1=\omega_{2,e}\beta_2=\frac{\beta_{e,H}}2,
 \label{eq:sipdg-weight-identities}\\
 &\min\{\beta_1,\beta_2\}\leq\beta_{e,H}
       \leq2\min\{\beta_1,\beta_2\},
 \qquad
 \frac{\omega_{j,e}^2\beta_j}{\beta_{e,H}}\leq\frac12,
 \quad j=1,2.
 \label{eq:sipdg-weight-bounds}
\end{align}
In particular, denote
$q_e(v):=\aver{\beta\partial_{\bfn_e}v}_{\omega,e}$, then
\begin{align}
 \frac{|q_e(v)|^2}{\beta_{e,H}}
 &\leq\frac12\sum_{j=1}^2
                  \beta_j|\nabla v_j\cdot\bfn_e|^2,
 \label{eq:sipdg-pointwise-flux-bound}\\
 \beta_{e,H}|\jump{v}_e|^2
 &\leq4\sum_{j=1}^2\beta_j|v_j|^2.
 \label{eq:sipdg-pointwise-jump-bound}
\end{align}
Indeed, the first inequality is weighted Cauchy--Schwarz, using
$\sum_j\omega_{j,e}^2\beta_j=\beta_{e,H}/2$; the second follows
from $|v_1-v_2|^2\leq2(|v_1|^2+|v_2|^2)$ and
$\beta_{e,H}\leq2\beta_j$. Boundary analogues hold with a single
trace and constant one. None of these bounds involves the coefficient contrast.

Next, we set the penalty parameter $\gamma_e$ also using the harmonic average $\beta_{e,H}$:
\begin{equation}
 \gamma_e:=m(m+1)\beta_{e,H},
 \qquad
 \lambda_e:=\frac{\sigma_0\gamma_e}{h_e},
 \label{eq:sipdg-harmonic-penalty}
\end{equation}
where $\sigma_0>0$ is chosen below. For $w,v\in H^2(\mathcal T_h)$,
define
\begin{align}
 a_h(w,v)
 &:=\sum_{K\in\mathcal T_h}\beta_K(\nabla w,\nabla v)_K
    -\sum_{e\in\mathcal E_h}\langle q_e(w),\jump{v}_e\rangle_e
    \nonumber\\
 &\quad-\sum_{e\in\mathcal E_h}\langle q_e(v),\jump{w}_e\rangle_e
    +\sum_{e\in\mathcal E_h}
           \lambda_e\langle\jump{w}_e,\jump{v}_e\rangle_e,
 \label{eq:sipdg-bilinear}\\
 L_h(v)
 &:=(f,v)_\Omega+
    \sum_{e\in\mathcal E_h^b}
          \langle g,-\beta\partial_{\bfn_e}v+\lambda_e v\rangle_e.
 \label{eq:sipdg-linear-functional}
\end{align}
Recall the global  GC-FE space $S_h^m(\mathcal T_h)$ defined in \eqref{eq:GCFE_space_Omega}. The SIPDG-GC-FE method seeks 
$u_h\in S_h^m(\mathcal T_h)$ such that
\begin{equation}
 a_h(u_h,v_h)=L_h(v_h)
 \qquad\forall v_h\in S_h^m(\mathcal T_h).
 \label{eq:sipdg-discrete}
\end{equation}
Both flux terms use the same harmonic average; the jump itself remains
the ordinary difference of traces. The term involving $g$ imposes the
Dirichlet boundary condition \eqref{eq:interface-bc} weakly, without requiring $g$ to belong to a discrete
trace space. 

\subsection{Coefficient-robust stability and consistency}
\label{subsec:sipdg-error-analysis}

Define the energy norm on $H^1(\mathcal T_h)$ and the augmented norm
on $H^2(\mathcal T_h)$ by
\begin{align}
 \norm{v}_{h,\beta}^2
 &:=\sum_{K\in\mathcal T_h}\beta_K\abs{v}_{1,K}^2
     +\sum_{e\in\mathcal E_h}
                       \lambda_e\norm{\jump{v}_e}_{0,e}^2,
 \label{eq:sipdg-energy-norm}\\
 \normtvb{v}_{h,\beta}^2
 &:=\norm{v}_{h,\beta}^2+
       \sum_{e\in\mathcal E_h}\lambda_e^{-1}\norm{q_e(v)}_{0,e}^2.
 \label{eq:sipdg-augmented-norm}
\end{align}
It is straightforward to verify that $\norm{\cdot}_{h,\beta}$ is a norm on $H^1(\mathcal T_h)$:
zero elementwise gradients and zero interior jumps give a global
constant, and the boundary terms force that constant to vanish.
For brevity, write $V(v)^2$, $J(v)^2$, and $F(v)^2$ for the volume,
jump, and flux sums in these definitions, respectively. That is 
\begin{equation}
V(v)^2:=\sum_{K\in\mathcal T_h}\beta_K\abs{v}_{1,K}^2,\quad
J(v)^2:=\sum_{e\in\mathcal E_h}
                       \lambda_e\norm{\jump{v}_e}_{0,e}^2,\quad
F(v)^2:= \sum_{e\in\mathcal E_h}\lambda_e^{-1}\norm{q_e(v)}_{0,e}^2.                      
\end{equation}

Unless otherwise specified, $C$ may depend on the fixed degree $m$,
the fixed penalty parameter $\sigma_0$, and the mesh and geometric
regularity constants. It is independent of $h$, both coefficient
values $\beta^-$ and $\beta^+$, and their ratio.

\begin{lemma}[Weighted discrete flux estimate]
\label{lem:sipdg-global-inverse-trace}
There is a constant $C_{\rm it}>0$, independent of $h$, $\sigma_0$,
and the coefficient contrast, such that
\begin{equation}
 F(v_h)^2\leq\frac{C_{\rm it}}{\sigma_0}V(v_h)^2
 \qquad\forall v_h\in S_h^m(\mathcal T_h).
 \label{eq:sipdg-global-inverse-trace}
\end{equation}
Here $C_{\rm it}$ may depend on $m$ and the mesh and
geometric regularity constants.
\end{lemma}
\begin{proof}
Use \eqref{eq:sipdg-pointwise-flux-bound} on interior edges and its
one-sided version on boundary edges. Since $h_e\simeq h_K$, summing
over the edges and using \eqref{eq:sipdg-harmonic-penalty} gives
\begin{align*}
 F(v_h)^2
 &\leq\frac{C}{\sigma_0m(m+1)}
       \sum_{K\in\mathcal T_h}
           h_K\beta_K\norm{\nabla v_h}_{0,\partial K}^2\\
 &\leq\frac{C}{\sigma_0m(m+1)}
       \sum_{K\in\mathcal T_h}\beta_K\abs{v_h}_{1,K}^2\\
 &=\frac{C}{\sigma_0m(m+1)}
       \sum_{K\in\mathcal T_h}
                    \norm{\beta^{1/2}\nabla v_h}_{0,K}^2.
\end{align*}
The second inequality follows from the GC-FE discrete gradient trace estimate
\eqref{eq:GCFE_discrete_gradient_trace}; the last equality is due to $\beta_K$ being a constant on $K$.
This proves the assertion without comparing coefficients across an
interface. The factor $m(m+1)$ is a penalty scaling factor and the
estimates of Section~\ref{sec:GCFE_spaces} permit dependence on the
fixed degree.
\end{proof}

\begin{lemma}[Boundedness, coercivity, and well-posedness]
\label{th:sipdg-stability-wellposedness}
For $w,v\in H^2(\mathcal T_h)$,
\begin{equation}
 |a_h(w,v)|\leq3\normtvb{w}_{h,\beta}\normtvb{v}_{h,\beta}.
 \label{eq:sipdg-boundedness}
\end{equation}
Moreover,
\begin{equation}
 |a_h(w,v_h)|
 \leq3\left(1+\frac{C_{\rm it}}{\sigma_0}\right)^{1/2}
       \normtvb{w}_{h,\beta}\norm{v_h}_{h,\beta},
 \quad w\in H^2(\mathcal T_h),\quad v_h\in S_h^m(\mathcal T_h).
 \label{eq:sipdg-one-sided-boundedness}
\end{equation}
If $\sigma_0\geq\sigma_*:=4C_{\rm it}$, then
\begin{equation}
 a_h(v_h,v_h)\geq\tfrac12\norm{v_h}_{h,\beta}^2
 \qquad\forall v_h\in S_h^m(\mathcal T_h).
 \label{eq:sipdg-coercivity}
\end{equation}
Consequently, \eqref{eq:sipdg-discrete} has a unique solution. The
penalty and stability constants are independent of the jump.
\end{lemma}
\begin{proof}
Cauchy--Schwarz in the volume and edge terms yields
\[
 |a_h(w,v)|\leq V(w)V(v)+J(w)J(v)+F(w)J(v)+F(v)J(w).
\]
The sum of the first two terms is bounded by
$\norm{w}_{h,\beta}\norm{v}_{h,\beta}$; each remaining term is
bounded by the product of augmented norms. This proves
\eqref{eq:sipdg-boundedness}. Lemma~\ref{lem:sipdg-global-inverse-trace}
also gives
\begin{equation}
 \normtvb{v_h}_{h,\beta}
 \leq\left(1+\frac{C_{\rm it}}{\sigma_0}\right)^{1/2}
          \norm{v_h}_{h,\beta},
 \label{eq:sipdg-discrete-norm-equivalence}
\end{equation}
which proves \eqref{eq:sipdg-one-sided-boundedness}. Finally,
\begin{align*}
 a_h(v_h,v_h)
 &\geq V(v_h)^2+J(v_h)^2-2F(v_h)J(v_h)\\
 &\geq\tfrac12 V(v_h)^2+
             \left(1-\frac{2C_{\rm it}}{\sigma_0}\right)J(v_h)^2,
\end{align*}
where $2ab\leq a^2/2+2b^2$ was used after
\eqref{eq:sipdg-global-inverse-trace}. This implies coercivity for
$\sigma_0\geq4C_{\rm it}$. Existence and uniqueness follow from
coercivity on the finite-dimensional space.
\end{proof}

\begin{lemma}[Consistency]
\label{lem:sipdg-consistency}
Let $u\in H^1(\Omega)\cap\mathcal H^2(\Omega)$ solve
\eqref{eq:interface-pde}--\eqref{eq:jumps}. Then
\begin{equation}
 a_h(u,v_h)=L_h(v_h),
 \qquad
 a_h(u-u_h,v_h)=0
 \quad\forall v_h\in S_h^m(\mathcal T_h).
 \label{eq:sipdg-galerkin-orthogonality}
\end{equation}
\end{lemma}
\begin{proof}
The normal flux $\beta\nabla u\cdot\bfn_e$ is single-valued on
ordinary interior edges, and the interface flux jump condition
\eqref{eq:jumps} gives the same property on
$\mathcal E_h^\Gamma$. Elementwise integration by parts and
\eqref{eq:sipdg-weighted-product-rule} therefore give
\[
 \sum_{K\in\mathcal T_h}(\beta\nabla u,\nabla v_h)_K
 -\sum_{e\in\mathcal E_h}\langle q_e(u),\jump{v_h}_e\rangle_e
 =(f,v_h)_\Omega.
\]
Here the flux jump term vanishes and the harmonic average of a
single-valued flux equals that flux. Since $\jump{u}_e=0$ on
$\mathcal E_h^\circ$ and $\jump{u}_e=g$ on $\mathcal E_h^b$,
the other two terms in $a_h(u,v_h)$ give exactly the boundary terms
in $L_h(v_h)$. Subtracting the discrete equation proves orthogonality.
All integrations use $\Gamma_b$, $\Gamma_i$, and their exact normals;
no replacement of a curved interface or boundary is involved.
\end{proof}

\subsection{Weighted approximation and energy error estimates}
\label{subsec:sipdg-robust-approximation}

For an element $K$, define its approximation domain by
\begin{equation}
 D_K:=
 \begin{cases}
 K_F^{s_K},&K\in\mathcal T_h^c,\\
 K,&K\in\mathcal T_h^r.
 \end{cases}
 \label{eq:sipdg-approximation-domain}
\end{equation}
As in Section~\ref{sec:GCFE_spaces}, $D_K\subset\Omega_K$ lies
entirely on the physical side containing $K$. In particular, no
projection or extension across the coefficient interface is needed.
For integers $2\leq r_K\leq m+1$, set
\begin{equation}
 \mathcal A_{h,\beta}(v;\{r_K\})^2
 :=\sum_{K\in\mathcal T_h}
       \beta_K h_K^{2(r_K-1)}\norm{v}_{r_K,D_K}^2.
 \label{eq:sipdg-local-approximation-quantity}
\end{equation}

\begin{lemma}[Weighted projection estimate]
\label{lem:sipdg-projection-skeleton}
\label{lem:sipdg-augmented-approximation}
Assume $v|_{D_K}\in H^{r_K}(D_K)$ for every $K$, where
$2\leq r_K\leq m+1$ are integers. Then
\begin{equation}
 \normtvb{v-\Pi_hv}_{h,\beta}
 \leq C\mathcal A_{h,\beta}(v;\{r_K\}).
 \label{eq:sipdg-augmented-approximation}
\end{equation}
In particular, put
\begin{equation}
 \norm{v}_{\mathcal H^r_\beta(\Omega)}^2
 :=\sum_{s\in\{-,+\}}\beta^s\norm{v}_{r,\Omega^s}^2.
 \label{eq:sipdg-weighted-sobolev-norm}
\end{equation}
For a common integer $2\leq r\leq m+1$ and
$v\in\mathcal H^r(\Omega)$,
\begin{equation}
 \normtvb{v-\Pi_hv}_{h,\beta}
 \leq C h^{r-1}\norm{v}_{\mathcal H^r_\beta(\Omega)}.
 \label{eq:sipdg-global-weighted-approximation}
\end{equation}
\end{lemma}
\begin{proof}
Write $\eta=v-\Pi_hv$. Corollary~\ref{cor:GCFE_DG_projection_estimate}
gives, for $j=0,1,2$,
\begin{equation}
 \abs{\eta}_{j,K}
 \leq C h_K^{r_K-j}\norm{v}_{r_K,D_K}.
 \label{eq:sipdg-local-projection-input}
\end{equation}
The general and gradient trace estimates
\eqref{eq:TraceIneq} and \eqref{eq:GCFE_gradient_trace} imply
\begin{align}
 \norm{\eta}_{0,\partial K}^2
 &\leq C h_K^{2r_K-1}\norm{v}_{r_K,D_K}^2,
 \label{eq:sipdg-projection-value-trace}\\
 \norm{\nabla\eta}_{0,\partial K}^2
 &\leq C h_K^{2r_K-3}\norm{v}_{r_K,D_K}^2.
 \label{eq:sipdg-projection-gradient-trace}
\end{align}
Consequently,
\[
 V(\eta)^2
 =\sum_K\beta_K\abs{\eta}_{1,K}^2
 \leq C\mathcal A_{h,\beta}(v;\{r_K\})^2.
\]
By \eqref{eq:sipdg-pointwise-jump-bound}, its boundary analogue,
and $h_e\simeq h_K$,
\begin{align*}
 J(\eta)^2
 &\leq C\sigma_0m(m+1)
           \sum_K\beta_K h_K^{-1}\norm{\eta}_{0,\partial K}^2 \leq C\sigma_0m(m+1)\mathcal A_{h,\beta}(v;\{r_K\})^2.
\end{align*}
Likewise, \eqref{eq:sipdg-pointwise-flux-bound} gives
\begin{align*}
 F(\eta)^2
 &\leq\frac{C}{\sigma_0m(m+1)}
            \sum_K\beta_K h_K\norm{\nabla\eta}_{0,\partial K}^2\leq\frac{C}{\sigma_0m(m+1)}
                       \mathcal A_{h,\beta}(v;\{r_K\})^2.
\end{align*}
This proves \eqref{eq:sipdg-augmented-approximation} for fixed
$m$ and $\sigma_0$. Each sum is weighted by the coefficient of
the same physical subdomain as its approximation error.

For the global estimate, $\beta_K=\beta^s$ whenever
$K\subset\Omega^s$. The regular elements have disjoint interiors,
and Lemma~\ref{lem:bnd_for_sum_of_chi_{K_F}(X)} gives uniformly
finite overlap of the fictitious elements. Since $D_K\subset\Omega_K$,
we can sum within each physical subdomain to obtain
\[
 \sum_K\beta_K\norm{v}_{r,D_K}^2
 \leq C\sum_{s\in\{-,+\}}\beta^s\norm{v}_{r,\Omega^s}^2.
\]
Using $h_K\leq h$ proves
\eqref{eq:sipdg-global-weighted-approximation}.
\end{proof}

\begin{theorem}[Robust quasi-optimality and energy error]
\label{th:sipdg-energy-error}
Let $\sigma_0\geq\sigma_*$ and let
$u\in H^1(\Omega)\cap\mathcal H^2(\Omega)$ solve
\eqref{eq:interface-pde}--\eqref{eq:jumps}. The solution of
\eqref{eq:sipdg-discrete} satisfies
\begin{equation}
 \normtvb{u-u_h}_{h,\beta}
 \leq C\inf_{v_h\in S_h^m(\mathcal T_h)}
                    \normtvb{u-v_h}_{h,\beta}.
 \label{eq:sipdg-robust-quasioptimality}
\end{equation}
If in addition $u|_{D_K}\in H^{r_K}(D_K)$ with integers
$2\leq r_K\leq m+1$, then
\begin{equation}
 \norm{u-u_h}_{h,\beta}
 \leq\normtvb{u-u_h}_{h,\beta}
 \leq C\left(\sum_{K\in\mathcal T_h}
          \beta_K h_K^{2(r_K-1)}\norm{u}_{r_K,D_K}^2\right)^{1/2}.
 \label{eq:sipdg-local-energy-error}
\end{equation}
In particular, if $u\in\mathcal H^{m+1}(\Omega)$, then
\begin{equation}
 \norm{u-u_h}_{h,\beta}
 \leq\normtvb{u-u_h}_{h,\beta}
 \leq C h^m\norm{u}_{\mathcal H^{m+1}_\beta(\Omega)}.
 \label{eq: energy bound}
\end{equation}
Equivalently,
\begin{equation}
 \norm{u-u_h}_{h,\beta}
 \leq C h^m\left(
    \beta^-\norm{u}_{m+1,\Omega^-}^2+
    \beta^+\norm{u}_{m+1,\Omega^+}^2\right)^{1/2},
 \label{eq:sipdg-contrast-robust-energy-error}
\end{equation}
with $C$ independent of both coefficient values and their ratio.
\end{theorem}
\begin{proof}
Fix $v_h\in S_h^m(\mathcal T_h)$ and set $\xi_h=v_h-u_h$.
Coercivity, consistency, and one-sided boundedness yield
\begin{align*}
 \tfrac12\norm{\xi_h}_{h,\beta}^2
 &\leq a_h(\xi_h,\xi_h)
   =a_h(v_h-u,\xi_h)\leq C\normtvb{u-v_h}_{h,\beta}\norm{\xi_h}_{h,\beta}.
\end{align*}
Thus, 
\[
 \norm{\xi_h}_{h,\beta}\leq C\normtvb{u-v_h}_{h,\beta}.
\]
By \eqref{eq:sipdg-discrete-norm-equivalence} and the triangle
inequality,
\[
 \normtvb{u-u_h}_{h,\beta}
 \leq\normtvb{u-v_h}_{h,\beta}+\normtvb{\xi_h}_{h,\beta}
 \leq C\normtvb{u-v_h}_{h,\beta}.
\]
Taking the infimum proves \eqref{eq:sipdg-robust-quasioptimality}.
Choosing $v_h=\Pi_hu$ and applying the weighted projection lemma
proves \eqref{eq:sipdg-local-energy-error} and \eqref{eq: energy bound}.
The last assertion follows by expanding the weighted Sobolev norm.
\end{proof}

\begin{remark}[Meaning and scope of robustness]
The coefficient factors on the right-hand side of
\eqref{eq:sipdg-contrast-robust-energy-error} are part of the natural
weighted solution norm; they are not hidden in $C$. The solution
itself may depend on the coefficients. No quasi-monotonicity condition
on the coefficient distribution is used, since the projection estimates
are summed separately within physical subdomains. The full Sobolev
norms agree with Section~\ref{sec:GCFE_spaces}: under a nonlinear
Frenet transformation, derivatives of the pullback contain lower-order
physical derivatives, and Cartesian polynomials need not belong to
$S_h^m(K)$ on curved elements.

For the unweighted broken $H^1$ seminorm one only obtains, by norm
conversion,
\[
 \abs{u-u_h}_{1,\mathcal T_h}
 \leq\beta_*^{-1/2}\norm{u-u_h}_{h,\beta}.
\]
Thus the robustness assertion concerns the weighted energy norm.
The fitted-mesh proof does not assert contrast-independent estimates
for GC-IFE functions on interface-cut elements.
\end{remark}


\subsection{The $L^2$ error and the role of dual regularity}
\label{subsec:sipdg-L2-analysis}

Given $\psi\in L^2(\Omega)$, let $z\in H_0^1(\Omega)$ solve
\begin{equation}
 \begin{aligned}
 -\nabla\cdot(\beta\nabla z)&=\psi
       &&\text{in }\Omega^-\cup\Omega^+,\\
 z&=0&&\text{on }\Gamma_b,\\
 \jump{z}_{\Gamma_i}=0,\qquad
 \jump{\beta\partial_{\bfn}z}_{\Gamma_i}&=0
       &&\text{on }\Gamma_i.
 \end{aligned}
 \label{eq:sipdg-dual-problem}
\end{equation}
Assume that $z\in\mathcal H^2(\Omega)$ and that
\begin{equation}
 \norm{z}_{\mathcal H^2_\beta(\Omega)}
 \leq C_{\rm reg}(\beta)\norm{\psi}_{0,\Omega}.
 \label{eq:sipdg-dual-regularity}
\end{equation}
The coefficient dependence of this continuous regularity constant is
kept explicit; it is separate from the robust discrete estimates.

\begin{theorem}[$L^2$ error estimate]
\label{th:sipdg-L2-error}
Assume the hypotheses of Theorem~\ref{th:sipdg-energy-error} and
$u\in\mathcal H^{m+1}(\Omega)$. If the dual regularity estimate
\eqref{eq:sipdg-dual-regularity} holds, then
\begin{equation}
 \norm{u-u_h}_{0,\Omega}
 \leq C C_{\rm reg}(\beta)h^{m+1}
                  \norm{u}_{\mathcal H^{m+1}_\beta(\Omega)}.
 \label{eq:sipdg-L2-error-explicit-regularity}
\end{equation}
Here $C$ is independent of $h$ and of the diffusion
coefficients. The constant $C_{\mathrm{reg}}(\beta)$
comes from the continuous dual regularity estimate and
may depend on both diffusion coefficients and their contrast.

\end{theorem}
\begin{proof}
Set $e_h=u-u_h$ and take $\psi=e_h$ in
\eqref{eq:sipdg-dual-problem}. The dual solution has zero value
jumps on every interior edge and zero trace on the boundary, while
its normal flux is single-valued on interior edges. Elementwise
integration by parts therefore gives
\[
 a_h(w,z)=(w,e_h)_\Omega
 \qquad\forall w\in H^2(\mathcal T_h).
\]
By Galerkin orthogonality,
\[
 \norm{e_h}_{0,\Omega}^2
 =a_h(e_h,z)=a_h(e_h,z-\Pi_hz).
\]
Using boundedness, \eqref{eq: energy bound}, and
\eqref{eq:sipdg-global-weighted-approximation} with $r=2$ gives
\begin{align*}
 \norm{e_h}_{0,\Omega}^2
 &\leq3\normtvb{e_h}_{h,\beta}\normtvb{z-\Pi_hz}_{h,\beta}\\
 &\leq C h^{m+1}\norm{u}_{\mathcal H^{m+1}_\beta(\Omega)}
                         \norm{z}_{\mathcal H^2_\beta(\Omega)}\\
 &\leq C C_{\rm reg}(\beta)h^{m+1}
       \norm{u}_{\mathcal H^{m+1}_\beta(\Omega)}\norm{e_h}_{0,\Omega}.
\end{align*}
If $e_h=0$, the assertion is immediate. Otherwise,
dividing by $\|e_h\|_{0,\Omega}$ proves the stated estimate.

\end{proof}

\begin{remark}
The coefficient independence of the discrete energy estimate
does not imply a contrast-independent regularity estimate
for the continuous dual problem in the full weighted Sobolev
norm. We therefore retain \(C_{\mathrm{reg}}(\beta)\)
explicitly in the \(L^2\) error estimate. Under the stated
dual regularity assumption, this estimate gives the optimal
convergence order \(O(h^{m+1})\) for fixed diffusion
coefficients. No uniformity with respect to the coefficient
contrast is asserted for this \(L^2\) estimate.
\end{remark}

\begin{remark}[Extension to variable coefficients]
\label{rem:sipdg-variable-coefficients}
The method extends to coefficients
$\beta^\pm\in C^0(\overline{\Omega^\pm})$ that are strictly positive
on their respective closures. Assume that
\[
 \max_{K\in\mathcal T_h}
 \frac{\max_{\overline K}\beta}{\min_{\overline K}\beta}
 \leq C_\beta,
\]
where the extrema use the coefficient from the material containing
$K$, and $C_\beta$ is independent of $h$ and the interface jump.

Evaluate the harmonic weights \eqref{eq:sipdg-harmonic-weights}
pointwise and retain the variable weights inside all volume and edge
integrals, including those defining the norms. The stability and
approximation arguments extend with the penalty threshold and constants
allowed to depend on $C_\beta$ but independent of the interface jump.
In the approximation bounds, replace $\beta_K$ and $\beta^s$ by
$\max_{\overline K}\beta$ and
$\max_{\overline{\Omega^s}}\beta^s$, respectively. The optimal-order
energy estimate retains the stated piecewise $H^{m+1}$ regularity
assumption; the $L^2$ estimate additionally requires the corresponding
weighted dual regularity. These regularity properties do not follow
from coefficient continuity alone.
\end{remark}

\subsection{GC-FE approximation on a curved domain}
\label{subsec:curved-boundary-problem}

For the boundary value problem \eqref{eq:bvp-pde}--\eqref{eq:bvp-bc},
set $\Gamma_i=\emptyset$ and use the same formulation with a single
physical domain and a constant coefficient $\beta=\beta_0>0$.

\begin{corollary}
\label{cor:CB_error_estimates}
Assume $\Gamma_b$ is of class $C^{m+2}$, the mesh hypotheses of
Section~\ref{sec:GCFE_spaces}, and $\sigma_0\geq\sigma_*$. For
$u\in H^{m+1}(\Omega)$ solving
\eqref{eq:bvp-pde}--\eqref{eq:bvp-bc},
\begin{equation}
 \beta_0^{-1/2}\norm{u-u_h}_{h,\beta}
 \leq\beta_0^{-1/2}\normtvb{u-u_h}_{h,\beta}
 \leq C h^m\norm{u}_{m+1,\Omega}.
 \label{eq:sipdg-curved-domain-energy}
\end{equation}
If the homogeneous Dirichlet Laplacian has the usual $H^2$ regularity,
then
\begin{equation}
 \norm{u-u_h}_{0,\Omega}
 \leq C h^{m+1}\norm{u}_{m+1,\Omega}.
 \label{eq:sipdg-curved-domain-L2}
\end{equation}
The constants in both estimates are independent of $\beta_0$.
\end{corollary}
\begin{proof}
Apply the projection argument with $D_K=K_F^{s_K}\subset\Omega$
on curved-boundary elements and $D_K=K$ on regular elements.
The finite-overlap bound gives
$\normtvb{u-\Pi_hu}_{h,\beta}\leq
C h^m\beta_0^{1/2}\norm{u}_{m+1,\Omega}$.
The stability, consistency, and duality proofs above then apply
without an interface. For $\beta=\beta_0$, the dual problem becomes
$-\Delta z=\psi/\beta_0$, so Laplacian regularity gives
$C_{\rm reg}(\beta_0)\leq C\beta_0^{-1/2}$; the coefficient
factors in the $L^2$ estimate cancel.
\end{proof}

\section{Numerical Experiments}\label{sec:numerics}
In this section, we present numerical experiments to demonstrate computational features of GC-FE spaces. We
first  \yhred{compare }the conditioning of the local mass matrices  \yhred{associated with
the original and reconstructed bases}. We then proceed to show the performance of the SIPDG-GC-FE method. We examine convergence,
coefficient robustness, and exact geometry representation for interface
problems on fitted meshes, followed by convergence tests on a curved domain.
Finally, we demonstrate the performance of the SIPDG method that couples the GC-FE and GC-IFE spaces on meshes fitted to the outer boundary but unfitted to the internal interface.
Integrals over a curved element $K\in\mathcal{T}_h^c$ are computed with a
standard Gaussian quadrature rule on the reference triangle $\widehat K$,
transferred to $K$ through the blending map $F_K$ of Section~3.5.
Integrals over a curved edge are computed in the same way. Because the GC-FE basis functions are nonpolynomial on curved elements, these rules are not exact. We therefore use rules of sufficiently high degree that the quadrature error is negligible relative to the discretization error.
In all numerical experiments, we set the penalty parameter \(\sigma_0=\yhred{10}\) in the SIPDG method.

\subsection{Choice of a well-conditioned basis for GC-FE computations}
\label{sec:matrix_conditioning}

We compare the condition numbers of the quadrature-based local mass matrices
generated from the original GC-FE basis and the bases produced by RA1 and
RA2. Specifically, for each \(K\in\mathcal T_h^c\), we consider the three
local bases
\(
\boldsymbol\phi_K,
\quad
\widetilde{\boldsymbol\phi}_K^{(1)},
\quad
\widetilde{\boldsymbol\phi}_K^{(2)}
\)
introduced in Subsection~\ref{subsec:GC-FE Basis Reconstruction}, and  
their corresponding local matrices:
\[
\mathsf{M}_{q, K}(\boldsymbol\phi_K), \quad \mathsf{M}_{q, K}(\widetilde{\boldsymbol\phi}_K^{(1)}), \quad
\mathsf{M}_{q, K}(\widetilde{\boldsymbol\phi}_K^{(2)}).
\]

Table~\ref{tab:cond-mass} reports the maximum \(2\)-norm condition number
over all curved elements of a representative mesh\yhred{ with \(h\approx0.16836\)}. The condition number for
the original basis grows rapidly with the polynomial degree \(m\). Both
reconstructions substantially improve the conditioning, but the condition number of the mass matrix $\mathsf{M}_{q, K}(\widetilde{\boldsymbol\phi}_K^{(1)})$
begins to deteriorate at high degrees. By contrast, the condition number of the mass matrix $\mathsf{M}_{q, K}(\widetilde{\boldsymbol\phi}_K^{(2)})$ remains around $1$
throughout the tested range. We therefore use the basis
\(\widetilde{\boldsymbol\phi}_K^{(2)}\) produced by RA2 on every curved GC-FE element in the numerical examples below.

\begin{table}[htbp]
    \centering
    \caption{Maximum local mass-matrix condition numbers for the three GC-FE bases.}
    \label{tab:cond-mass}
    \small
    \begin{tabular}{cccc}
        \midrule
        \(m\) & \(\max\limits_{K\in\mathcal T_h^c}\left\{\operatorname{cond}_2(\mathsf{M}_{q, K}(\boldsymbol\phi_K))\right\}\) & \(\max\limits_{K\in\mathcal T_h^c}\left\{\operatorname{cond}_2(\mathsf{M}_{q, K}(\widetilde{\boldsymbol\phi}_K^{(1)}))\right\}\) & \(\max\limits_{K\in\mathcal T_h^c}\left\{\operatorname{cond}_2(\mathsf{M}_{q, K}(\widetilde{\boldsymbol\phi}_K^{(2)}))\right\}\) \\
        \midrule
        1  & \(2.8640\times10^{1}\)  & \(1.0000\) & \(1.0000\) \\
        2  & \(6.1557\times10^{2}\)  & \(1.0000\) & \(1.0000\) \\
        3  & \(1.3111\times10^{4}\)  & \(1.0000\) & \(1.0000\) \\
        4  & \(2.5122\times10^{5}\)  & \(1.0000\) & \(1.0000\) \\
        5  & \(1.2875\times10^{7}\)  & \(1.0000\) & \(1.0000\) \\
        6  & \(4.5506\times10^{8}\)  & \(1.0000\) & \(1.0000\) \\
        7  & \(1.7634\times10^{10}\) & \(1.0000\) & \(1.0000\) \\
        8  & \(5.8544\times10^{11}\) & \(1.0000\) & \(1.0000\) \\
        9  & \(2.2942\times10^{13}\) & \(1.0010\) & \(1.0000\) \\
        10 & \(7.7187\times10^{14}\) & \(1.6941\) & \(1.0000\) \\
        11 & \(4.7919\times10^{16}\) & \(55.330\) & \(1.0000\) \\
        12 & \(5.6628\times10^{18}\) & \(795.12\) & \(1.0000\) \\
        \bottomrule
    \end{tabular}
\end{table}

\subsection{Accuracy of the SIPDG method for interface problems}
\label{sec:gcfe_dg_interface}

This example  \yhred{examines the accuracy and robustness }of the SIPDG-GC-FE method
 \yhred{for interface problems on }interface-fitted meshes.

In the implementation of the SIPDG-GC-FE method, the finite element functions 
on straight-sided triangular elements are the standard ones such as those 
described by \citet{2008HesthavenWarburton}. The precursor
triangulations \(\overline{\mathcal T}_h\) are generated using the MATLAB PDE
Toolbox. We let 
\(h=\max\{
\operatorname{diam}(\overline K): {\overline K\in\overline{\mathcal T}_h}\}\) denote the maximum element diameter of
\(\overline{\mathcal T}_h\).

The test interface problem is posed on 
\(\Omega=\{(x,y):x^2/1.8^2+y^2/1.6^2<1\}\), 
which is an elliptical domain, and we let \(\Gamma_b=\partial\Omega\). In polar
coordinates \((r,\vartheta)\), define
\begin{equation}\label{eq: interface}
\varphi_1(r,\vartheta)
    =r^5\bigl(1+0.5\sin(6\vartheta)\bigr)^2-\frac{\pi}{3}.
\end{equation}
We then define the interface curve as the zero level set \(\Gamma_i=\{\varphi_1=0\}\), which is a six-lobed
flower-shaped curve. Let
\(\Omega^-=\{(x,y)\in\Omega:\varphi_1<0\}\) and
\(\Omega^+=\Omega\setminus\overline{\Omega^-}\), and let the piecewise
constant coefficient \(\beta\) satisfy
\(\beta|_{\Omega^\pm}=\beta^\pm>0\); see the left panel of  
Figure\yhred{~}\ref{fig:gcfe_dg_flower_mesh} for illustrations of $\Omega, \Gamma_b, \Gamma_i$ and the mesh. We then choose $f$ and $g$ in the interface problem \eqref{eq:interface-pde}\yhred{--}\eqref{eq:jumps} such that the exact solution is
\begin{equation}\label{eq:interface-example-solution}
u_1(x,y)=
    \begin{cases}
        \dfrac{\cos(\varphi_1(x,y))}{\beta^-},
        & (x,y)\in\Omega^-,
        \\[1ex]
        \dfrac{\cos(\varphi_1(x,y))}{\beta^+}
        +\dfrac{1}{\beta^-}-\dfrac{1}{\beta^+},
        & (x,y)\in\Omega^+.
    \end{cases}
\end{equation}

\begingroup
\raggedbottom

\subsubsection{Convergence of the SIPDG-GC-FE method}
\label{sec:gcfe_dg_flower_convergence}
To test the convergence of the SIPDG-GC-FE method, we use the weighted broken \(H^1\)-seminorm error
\[
\yhred{|u-u_h|_{1,h,\beta}
=\left(\sum_{K\in\mathcal T_h}\int_K\beta|\nabla(u-u_h)|^2\,dx\right)^{1/2},
}\]
\yhred{which is the volume component of the energy norm
\eqref{eq:sipdg-energy-norm}.
}

Table~\ref{tab:gcfe_dg_coefficient_fit} reports the convergence rates
of the SIPDG-GC-FE method when used to solve the interface problem on a sequence of fitted meshes
 \yhred{with }the following set of diffusion coefficients representing some typical jump contrasts:
\[
    (\beta^-,\beta^+)
    \in\left\{(1000,1),~(10,1),~(1,10),~(1,1000)\right\}.
\]
Figure~\ref{fig:gcfe_dg_interface_convergence} presents the \(L^2\)\yhred{-norm and
weighted }broken \(H^1\)-seminorm  \yhred{errors }for \yhred{\(1\le m\le 4\)} with
\((\beta^-,\beta^+)=(1,1000)\). The observed convergence rates for the \(L^2\)-norm and weighted broken \(H^1\)-seminorm errors, such as those reported in Figure~\ref{fig:gcfe_dg_interface_convergence} and 
Table~\ref{tab:gcfe_dg_coefficient_fit}, are consistent with the respective
\(O(h^{m+1})\) and \(O(h^m)\) estimates established in
Theorems~\ref{th:sipdg-energy-error}--\ref{th:sipdg-L2-error}.

 \begin{figure}[htbp]
    \centering
    \includegraphics[width=0.32\textwidth]
    {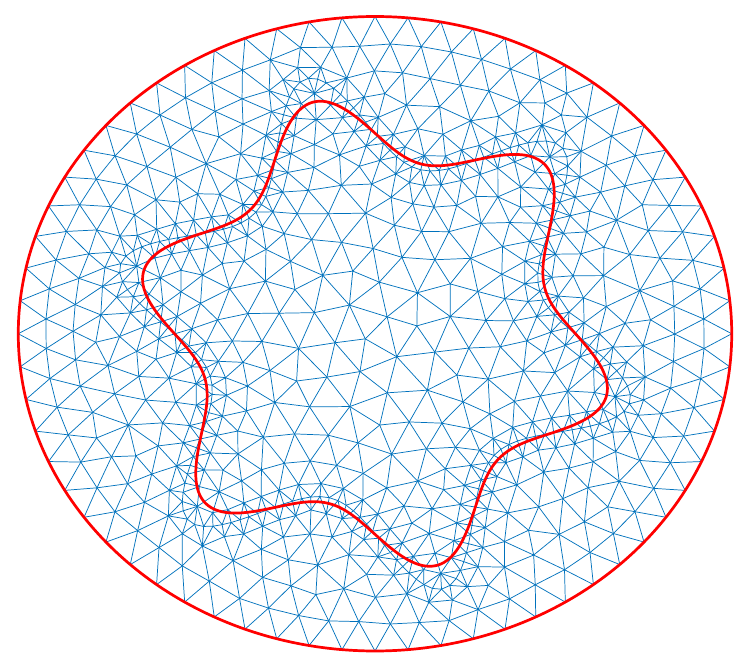}~
    \includegraphics[width=0.26\textwidth]
    {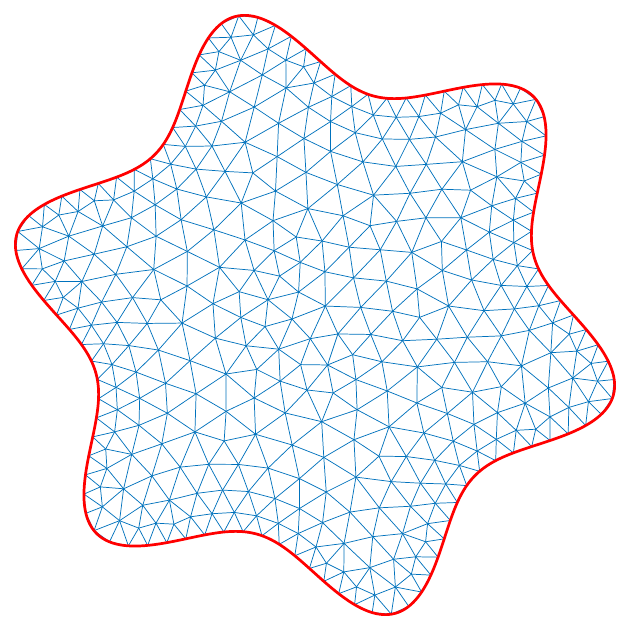}~
    \includegraphics[width=0.32\textwidth]
    {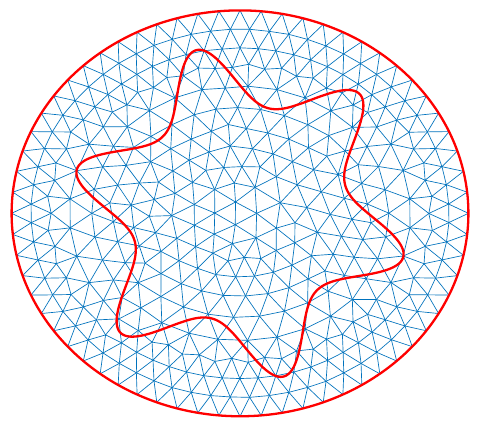}
    \caption{From left: a mesh fitted to the outer boundary and internal
    interface used in Subsection~\ref{sec:gcfe_dg_interface}; a
    boundary-fitted mesh for the flower-shaped domain used in
    Subsection~\ref{subsec:curved_domain_convergence}; and a boundary-fitted,
    interface-unfitted mesh used in
    Subsection~\ref{sec:gcfe_gcife_unfitted_app}.}
    \label{fig:gcfe_dg_flower_mesh}
\end{figure}

\begin{figure}[!htb]
  \centering
  \captionsetup[subfigure]{font=normalfont,labelfont=normalfont,textfont=normalfont,skip=2pt}
  \begin{subfigure}[t]{.49\textwidth}
    \centering
    \includegraphics[width=\linewidth]{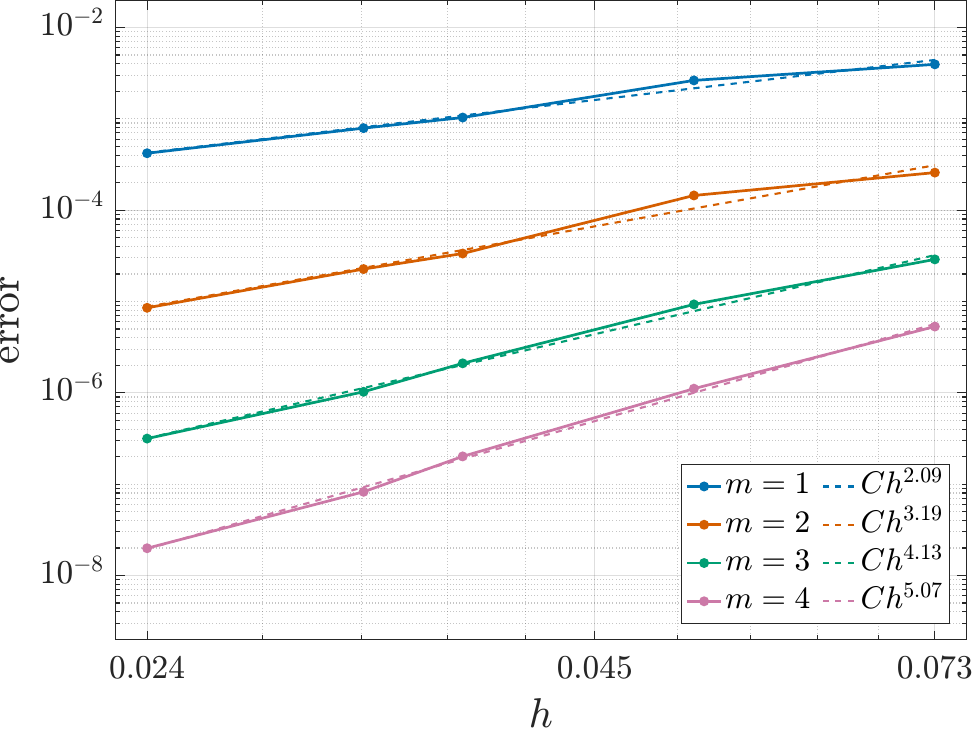}
    \caption{$L^2$-norm error.}
  \end{subfigure}\hfill
  \begin{subfigure}[t]{.49\textwidth}
    \centering
    \includegraphics[width=\linewidth]{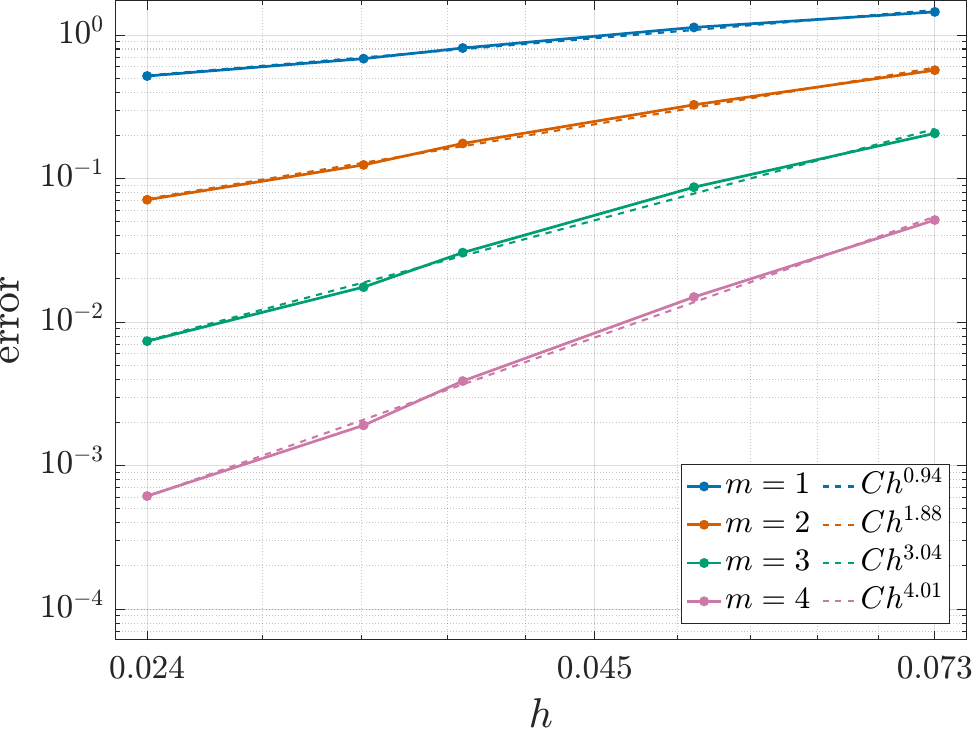}
    \caption{Weighted broken $H^1$-seminorm error.}
  \end{subfigure}
  \caption{Convergence of SIPDG-GC-FE solutions with
  $(\beta^-,\beta^+)=(1,1000)$ and $1\le m\le 4$.}
  \label{fig:gcfe_dg_interface_convergence}
\end{figure}

\begin{table}[!htb]
  \centering\small
  \caption{Convergence rates for SIPDG-GC-FE solutions with various coefficient contrasts.}
  \label{tab:gcfe_dg_coefficient_fit}
  \setlength{\tabcolsep}{4pt}
  \begin{tabular}{c*{8}{c}}
    \toprule
    $m$ & \multicolumn{4}{c}{$\|u-u_h\|_{L^2(\Omega)}$}
        & \multicolumn{4}{c}{$|u-u_h|_{1,h,\beta}$}\\
    \cmidrule(lr){2-5}\cmidrule(lr){6-9}
    & $(1000,1)$ & $(10,1)$ & $(1,10)$ & $(1,1000)$
    & $(1000,1)$ & $(10,1)$ & $(1,10)$ & $(1,1000)$\\
    \midrule
    1 & 1.83 & 1.83 & 1.83 & 2.09 & 0.93 & 0.93 & 0.93 & 0.94 \\
    2 & 2.94 & 2.94 & 2.94 & 3.19 & 1.88 & 1.88 & 1.88 & 1.88 \\
    3 & 4.09 & 4.09 & 4.09 & 4.13 & 3.04 & 3.04 & 3.04 & 3.04 \\
    4 & 5.07 & 5.07 & 5.07 & 5.07 & 4.01 & 4.01 & 4.01 & 4.01 \\
    \bottomrule
  \end{tabular}
\end{table}

\subsubsection{Comparison with arithmetic flux averaging}
\label{subsubsec:flower_raw_coefficient_contrast}

We denote the formulation \eqref{eq:sipdg-bilinear}, which uses weighted flux averages and harmonic penalization \citep{2009ErnStephansenZunino} with coefficient $\beta_{e,H}$, by W-H. For comparison, we consider the SIPDG formulation used in earlier GC-IFE studies \citep{2025AdjeridLinMeghaichi,2026LinLinZhang}, which employs arithmetic flux averages and penalization based on the maximum diffusion coefficient $\beta_{\max}$. We denote this formulation by A-max.

\commentout{
On each interior edge
\(e=\partial K_1\cap\partial K_2\), the latter uses

\[
\yhred{\aver{\beta\partial_{\bfn_e}v}_{A,e}
=\frac{1}{2}\left(\beta_1\partial_{\bfn_e}v_1
              +\beta_2\partial_{\bfn_e}v_2\right),
\qquad \gamma_{A,e}=m(m+1)\max(\beta_1,\beta_2),
}\]
\yhred{where \(\beta_j=\beta|_{K_j}\). The penalty weight is
\(\sigma_0\gamma_{A,e}/h_e\), following the maximum-coefficient scaling
used by \citet{2024AdjeridLinMeghaichi}. We denote this formulation by
A-max; both formulations use the same one-sided boundary terms}.
}

 \yhred{Both formulations are applied to the flower-interface problem with exact
solution \eqref{eq:interface-example-solution} and coefficient pair
\((\beta^-,\beta^+)=(1,\rho)\), where \(\rho=10^k\), \(k=0,\ldots,8\).
We fix \(m=5\) and \(h=3.2405\times10^{-2}\).
Both formulations use the same fitted mesh and RA2 basis.
Figure~\ref{fig:gcfe_formulation_robustness} presents the }\(L^2\)\yhred{-norm and
weighted }broken \(H^1\) \yhred{-seminorm errors.
}\begin{figure}[H]
\centering
 \captionsetup[subfigure]{font=normalfont,labelfont=normalfont,textfont=normalfont,skip=2pt}
\begin{subfigure}[t]{.49\textwidth}
\centering
 \includegraphics[width=\linewidth]{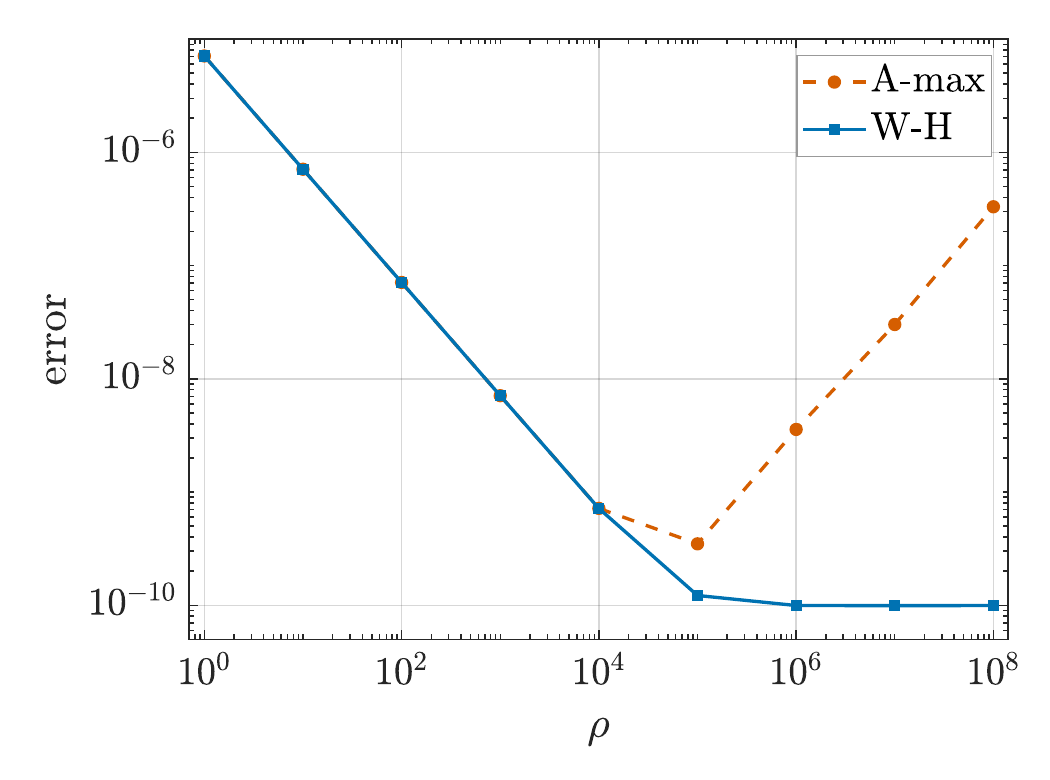}
\caption{\(L^2\)-norm error.}
\end{subfigure} \hfill
\begin{subfigure}[t]{.49\textwidth}
\centering
 \includegraphics[width=\linewidth]{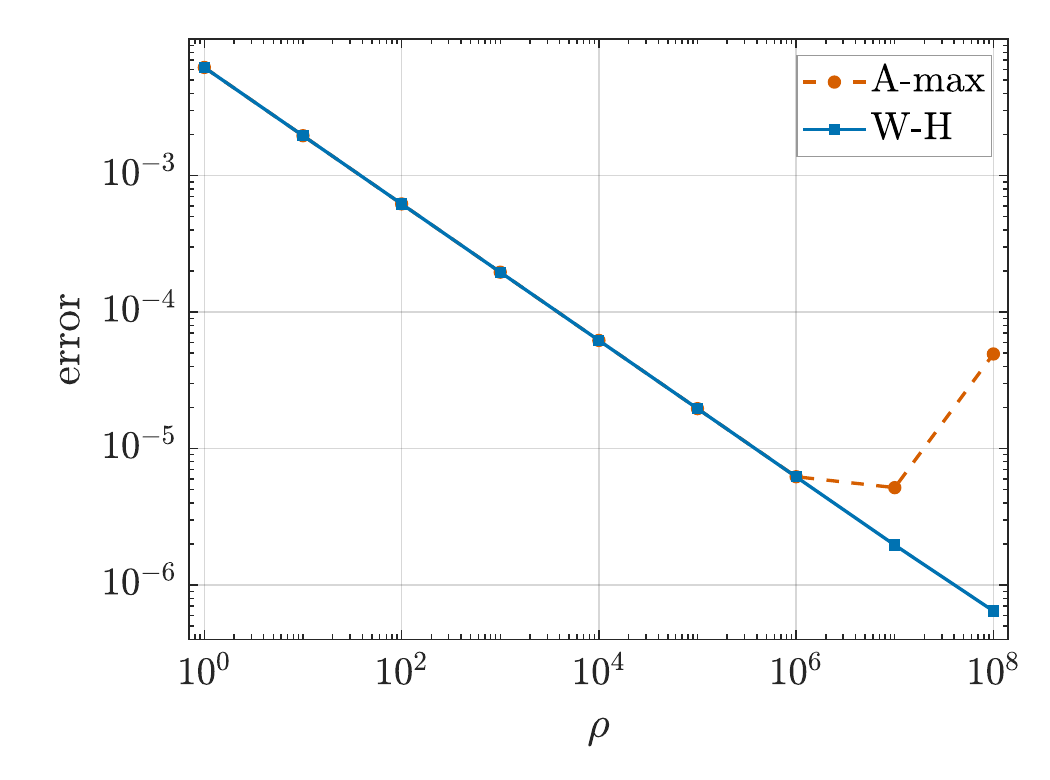}
\caption{ \yhred{Weighted broken }\(H^1\)-seminorm error.}
\end{subfigure}
\caption{ \yhred{Errors }of  \yhred{W-H }and  \yhred{A-max for }the  \yhred{flower-interface problem}.}
 \label{fig:gcfe_formulation_robustness}
\end{figure}

Figure~\ref{fig:gcfe_formulation_robustness} shows that the two formulations yield comparable errors at lower coefficient contrasts. However, the difference in their
\(L^2\)-norm errors gradually increases from \(\rho=10^4\) onward. In particular, the A-max formulation exhibits an increasingly pronounced loss of accuracy, whereas the W-H formulation maintains relatively low errors. Overall, these results indicate that the W-H formulation is more robust with respect to increasing coefficient contrast.

\subsection{Robustness with respect to the coefficient contrast}\label{sec:robust}
This example examines whether the constant $C$ in the energy error estimate
\eqref{eq:sipdg-contrast-robust-energy-error} is independent of the diffusion coefficients
$\beta^\pm$ and their ratio. We use the flower-shaped interface \eqref{eq: interface}, with the exact solution
\begin{equation}\label{eq:u2-robust}
u_2(x,y)=
\begin{cases}
2\beta^+\cos\!\big(\varphi_1(x,y)+\tfrac{\pi}{3}\big), & (x,y)\in\Omega^-,\\[2pt]
(\beta^-+\beta^+)\cos\!\big(\varphi_1(x,y)+\tfrac{\pi}{3}\big)
+(\beta^+-\beta^-)\cos\!\big(\varphi_1(x,y)-\tfrac{\pi}{3}\big), & (x,y)\in\Omega^+.
\end{cases}
\end{equation}

The exact solution $u_2$ has the nonzero interface flux
$\beta\nabla u_2\cdot\mathbf n=-\sqrt3\,\beta^-\beta^+\,\partial_{\mathbf n}\varphi_1$
on $\Gamma_i$, and its normal derivatives on the two sides differ by the
factor $\beta^+/\beta^-$. 

We take $h\approx0.04532$ and $1\le m\le4$, and consider the coefficient
pairs $(\beta^-,\beta^+)=(\rho,1)$ and $(1,\rho)$. Define the ratio
\begin{equation}\label{eq:Ch}
C_h(\rho)=\frac{\|u_2-u_h\|_{h,\beta}}
{h^m\big(\beta^-\|u_2\|_{m+1,\Omega^-}^2+\beta^+\|u_2\|_{m+1,\Omega^+}^2\big)^{1/2}}.
\end{equation}

The numerator is the error measured in the energy norm \eqref{eq:sipdg-energy-norm},
and the denominator is \(h^m\) times the weighted Sobolev norm defined in
\eqref{eq:sipdg-weighted-sobolev-norm}. For each fixed \(m\), the theoretical estimate predicts the bound
\(C_h(\rho)\le C\), with \(C\) independent of \(h\) and the coefficients.

Table~\ref{tab:gcfe_coefficient_ratio} shows that, for each fixed $m$, $C_h(\rho)$ remains close to a constant over $1\le\rho\le10^8$ for both $(\rho,1)$ and $(1,\rho)$. This is
consistent with the contrast-independent constant in \eqref{eq:sipdg-contrast-robust-energy-error}.


\begin{table}[!htb]
\centering
\caption{Computed ratios \(C_h(\rho)\) for the flower-interface problem with exact solution \eqref{eq:u2-robust} and \(h\approx0.04532\).}
\label{tab:gcfe_coefficient_ratio}
\small
\setlength{\tabcolsep}{4pt}
\begin{tabular}{cc*{5}{c}}
\toprule
\((\beta^-,\beta^+)\)&\(m\)&\(\rho=1\)&\(10^2\)&\(10^4\)&\(10^6\)&\(10^8\)\\
\midrule
    \multirow{4}{*}{\((\rho,1)\)} & 1 & \(1.76\!\times\!10^{-1}\) & \(1.77\!\times\!10^{-1}\) & \(1.77\!\times\!10^{-1}\) & \(1.77\!\times\!10^{-1}\) & \(1.77\!\times\!10^{-1}\)\\
     & 2 & \(1.42\!\times\!10^{-2}\) & \(1.42\!\times\!10^{-2}\) & \(1.42\!\times\!10^{-2}\) & \(1.42\!\times\!10^{-2}\) & \(1.42\!\times\!10^{-2}\)\\
     & 3 & \(7.18\!\times\!10^{-4}\) & \(7.05\!\times\!10^{-4}\) & \(7.05\!\times\!10^{-4}\) & \(7.05\!\times\!10^{-4}\) & \(7.05\!\times\!10^{-4}\)\\
     & 4 & \(3.05\!\times\!10^{-5}\) & \(3.16\!\times\!10^{-5}\) & \(3.16\!\times\!10^{-5}\) & \(3.16\!\times\!10^{-5}\) & \(3.16\!\times\!10^{-5}\)\\
\midrule
    \multirow{4}{*}{\((1,\rho)\)} & 1 & \(1.76\!\times\!10^{-1}\) & \(1.81\!\times\!10^{-1}\) & \(1.81\!\times\!10^{-1}\) & \(1.81\!\times\!10^{-1}\) & \(1.81\!\times\!10^{-1}\)\\
     & 2 & \(1.42\!\times\!10^{-2}\) & \(1.34\!\times\!10^{-2}\) & \(1.34\!\times\!10^{-2}\) & \(1.34\!\times\!10^{-2}\) & \(1.34\!\times\!10^{-2}\)\\
     & 3 & \(7.18\!\times\!10^{-4}\) & \(7.83\!\times\!10^{-4}\) & \(7.83\!\times\!10^{-4}\) & \(7.83\!\times\!10^{-4}\) & \(7.83\!\times\!10^{-4}\)\\
     & 4 & \(3.05\!\times\!10^{-5}\) & \(2.76\!\times\!10^{-5}\) & \(2.76\!\times\!10^{-5}\) & \(2.76\!\times\!10^{-5}\) & \(2.76\!\times\!10^{-5}\)\\
\bottomrule
\end{tabular}
\end{table}

\subsection{\yhred{Exact geometry representation in the GC-FE method}}
\label{subsubsec:gcfe_exact_geometry}

This experiment examines the individual roles and combined influence of the polynomial degree of the finite element space and the interface representation on numerical accuracy. A standard nodal polynomial isoparametric SIPDG method \citep[SIPDG-ISO-FE;][]{2008HesthavenWarburton} is included for comparison.

\yhred{Let \(\Omega=(-1,1)^2\), \(\Omega^-=\{x^2+y^2<R^2\}\), and
\(\Omega^+=\Omega\setminus\overline{\Omega^-}\), with \(R=1/2\) and
\(\Gamma_i=\partial\Omega^-\). }Both methods use the same straight-sided
precursor mesh with  \yhred{\(h\approx1.67\times10^{-1}\), the coefficient pair
\((\beta^-,\beta^+)=(1,1000)\), polynomial degrees \(1\le m\le 6\), and
the SIPDG formulation \eqref{eq:sipdg-bilinear}. GC-FE retains the exact physical
interface, whereas ISO-FE uses an isoparametric polynomial approximation.
The source term and boundary data are chosen so that the exact solution is
}\begin{equation}\yhred{\label{eq:circle-quadratic-solution}
u(x,y)=
\begin{cases}
1+\dfrac{2\beta^+}{\beta^-+\beta^+}(x^2+y^2-R^2),
& (x,y)\in\Omega^-,\\[1ex]
1+\dfrac{2\beta^-}{\beta^-+\beta^+}(x^2+y^2-R^2),
& (x,y)\in\Omega^+.
\end{cases}
}\end{equation}

The solution is quadratic in Cartesian coordinates on each subdomain.
Moreover, for the circle parameterization $\mathbf{g}(\xi)=R(\cos\xi, \sin\xi)^T$ with the outward unit normal $\mathbf{n}(\xi)=(\cos\xi,\sin\xi)^T$, the Frenet transformation gives $r=R+\eta$, and hence
\[
x^2+y^2-R^2=(R+\eta)^2-R^2=2R\eta+\eta^2.
\]
Thus, the solution in \eqref{eq:circle-quadratic-solution} is a quadratic polynomial in the Frenet coordinates on either side of the interface. Consequently, for $m \geq 2$
it can be represented exactly by the global GC-FE space, including on curved elements adjacent to the interface.

 \begin{figure}[H]
\centering
 \captionsetup[subfigure]{font=normalfont,labelfont=normalfont,textfont=normalfont,skip=2pt}
\begin{subfigure}[t]{.49\textwidth}
\centering
 \includegraphics[width=\linewidth]{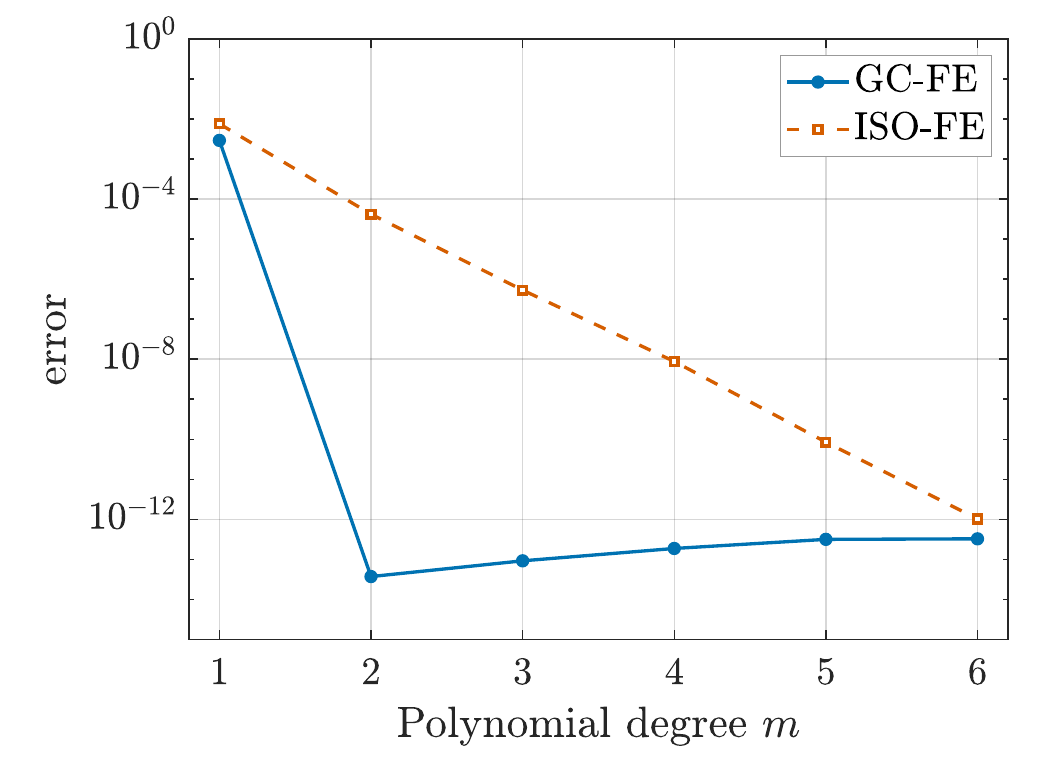}
\caption{ \yhred{\(L^2\)-norm error}.}
\end{subfigure} \hfill
\begin{subfigure}[t]{.49\textwidth}
\centering
 \includegraphics[width=\linewidth]{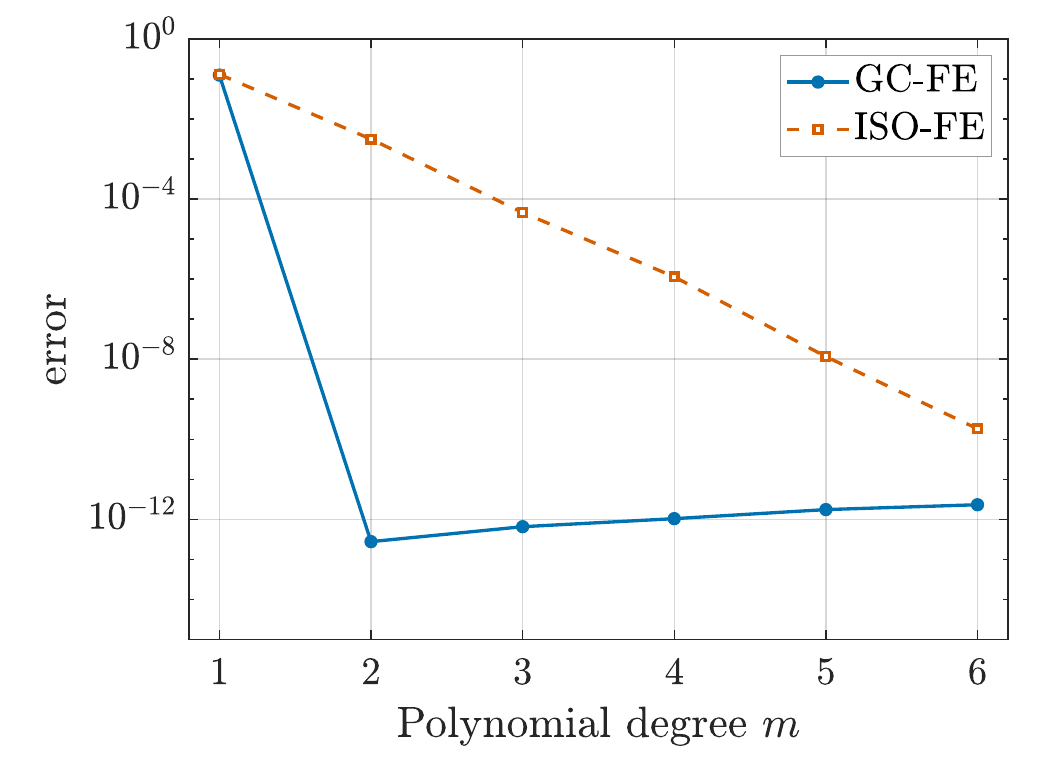}
\caption{ \yhred{Broken \(H^1\)-seminorm error}.}
\end{subfigure}
\caption{ \yhred{Errors of }SIPDG-GC-FE and SIPDG-ISO-FE \yhred{for the circular-interface
problem }with  \yhred{exact solution \eqref{eq:circle-quadratic-solution}}.}
 \label{fig:gcfe_iso_circle_degree}
\end{figure}

Figure~\ref{fig:gcfe_iso_circle_degree} shows that the SIPDG-GC-FE errors
reach roundoff level at $m=2$, reflecting the exact representation of the solution in the GC-FE space and the consistency of the formulation on the exact interface. On the contrary, Figure~\ref{fig:gcfe_iso_circle_degree} shows that the SIPDG-ISO-FE solution is obviously different from the exact solution even though its error decreases as $m$ increases. 

\endgroup
\subsection{Convergence of the SIPDG-GC-FE method on a curved domain}
\label{subsec:curved_domain_convergence}

This example demonstrates the use of the SIPDG-GC-FE method for a boundary
value problem on a curved domain. In polar coordinates \((r,\vartheta)\),
define
\(
\varphi_2(r,\vartheta)
=
r^4\bigl(1+0.3\sin(6\vartheta)\bigr)^2-\pi/3.
\)
We consider the boundary value problem
\eqref{eq:bvp-pde}--\eqref{eq:bvp-bc} on
\(
\Omega=\{\varphi_2<0\}, ~~\Gamma_b=\partial\Omega=\{\varphi_2=0\}.
\)
The domain and a corresponding mesh are shown in the middle panel of
Figure~\ref{fig:gcfe_dg_flower_mesh}. The diffusion coefficient is chosen as
\(
\beta(x,y)=1+x^2+y^2,
\)
and the source term \(f\) and boundary data \(g\) are chosen so that the
exact solution is
\(
u(x,y)=\cos(\pi x)\sin(\pi y).
\)

Figure~\ref{fig:gcfe_flower_boundary_convergence} presents the \(L^2\)-norm and
\yhred{weighted }broken \(H^1\)-seminorm errors of the SIPDG-GC-FE solutions computed in
\(S_h^m(\mathcal T_h)\) on a sequence of five successively refined meshes.
The observed convergence rates corroborate the
\(O(h^{m+1})\) and \(O(h^m)\) error estimates, respectively,  \yhred{discussed in
Remark~\ref{rem:sipdg-variable-coefficients}}.

Together with  \yhred{Remark~\ref{rem:sipdg-variable-coefficients}}, these numerical results
support the use of the GC-FE method as a geometry-exact alternative to
other higher-degree finite element methods for solving boundary value problems on curved domains.

\begin{figure}[htbp]
    \centering
\captionsetup[subfigure]{font=normalfont,labelfont=normalfont,textfont=normalfont,skip=2pt}
    \begin{subfigure}[t]{0.49\textwidth}
        \centering
        \includegraphics[width=\linewidth]
        {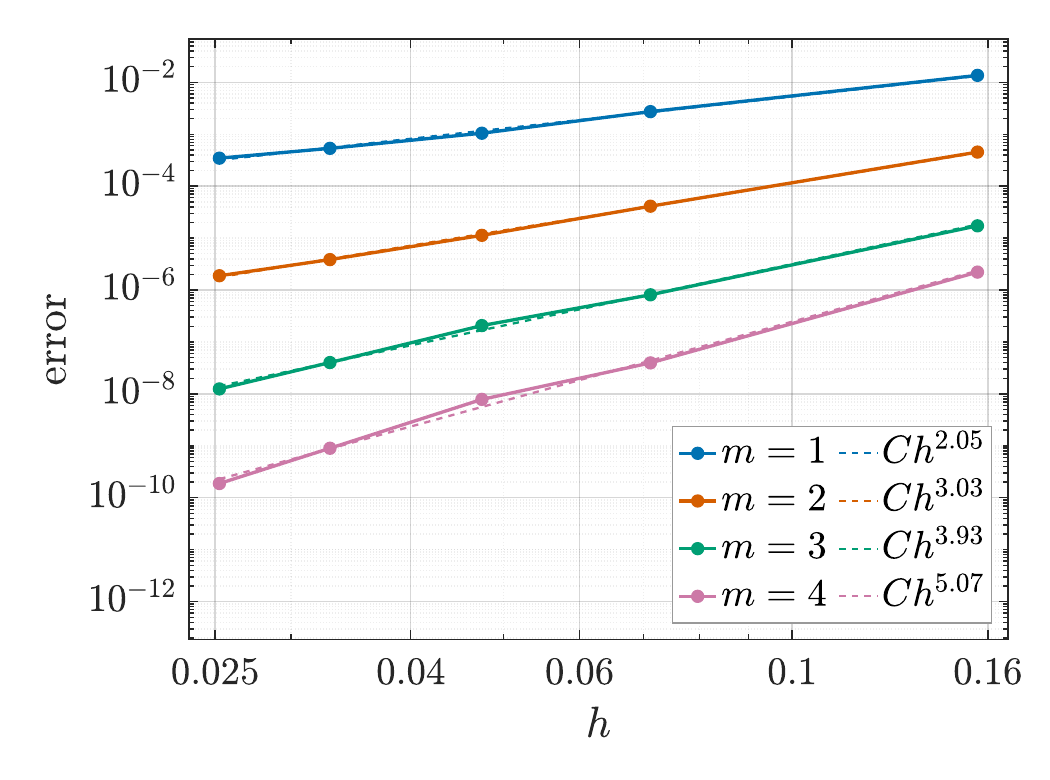}
        \caption{\(L^2\)-norm error.}
    \end{subfigure}\hspace{0.01\textwidth}
    \begin{subfigure}[t]{0.49\textwidth}
        \centering
        \includegraphics[width=\linewidth]
        {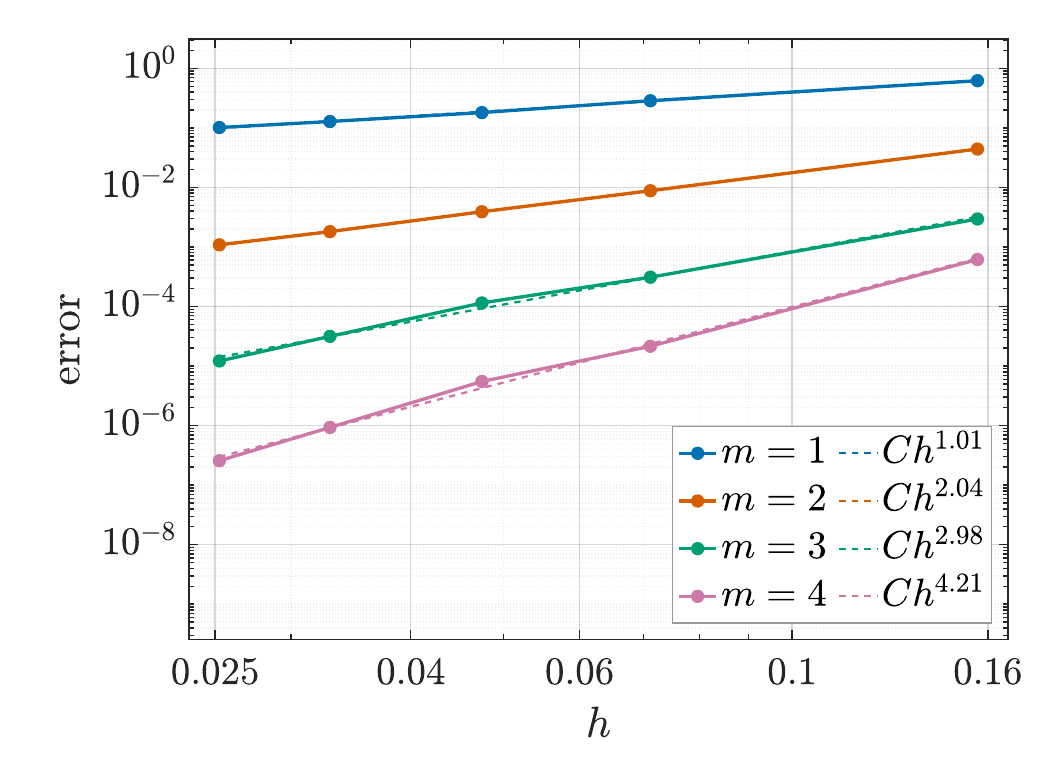}
        \caption{ \yhred{Weighted broken }\(H^1\)-seminorm error.}
    \end{subfigure}
    \caption{Convergence of the SIPDG-GC-FE method for the
variable-coefficient problem on the six-lobed flower-shaped domain.}
    \label{fig:gcfe_flower_boundary_convergence}
\end{figure}

\subsection{A high-order method to solve interface problems on unfitted meshes}
\label{sec:gcfe_gcife_unfitted_app}

This example demonstrates how GC-FE spaces can be coupled with
geometry-conforming immersed finite element (GC-IFE) spaces to solve the
interface problem \eqref{eq:interface-pde}--\eqref{eq:jumps} on a curved
domain using an interface-unfitted mesh. As described in
Section~\ref{sec:prelim}, such a mesh can be decomposed as
\[
\mathcal T_h
=
\mathcal T_h^i
\cup
\mathcal T_h^c
\cup
\mathcal T_h^r,
\qquad
\mathcal T_h^i\neq\emptyset.
\]
We define the following hybrid GC-FE-GC-IFE space:
\[
\widetilde S_h^m(\mathcal T_h)
=
\left\{
v_h\in L^2(\Omega):
v_h|_K\in
\begin{cases}
S_h^m(K), & K\in\mathcal T_h^c,\\
S_{h,\mathrm{IFE}}^m(K), & K\in\mathcal T_h^i,\\
\mathbb P_m(K), & K\in\mathcal T_h^r
\end{cases}
\quad\text{for every }K\in\mathcal T_h
\right\}.
\]
Here, \(S_{h,\mathrm{IFE}}^m(K)\) denotes the degree-\(m\) local GC-IFE
space developed by \citet{2026LinLinZhang}.  On interface elements and on edges cut by $\Gamma_i$, $\beta_K(\nabla w,\nabla v)_K$
   is replaced by $(\beta\nabla w,\nabla v)_K$, and the weights~(4.2) are evaluated
   pointwise. Since both traces on an interface-cut edge lie in the same material
   at each point, this gives $\omega_{1,e}=\omega_{2,e}=1/2$ and
   $\beta_{e,H}=\beta^\pm$ on $e\cap\Omega^\pm$.
   
   This unified
SIPDG-GC-FE-GC-IFE method is obtained from the SIPDG formulation
\eqref{eq:sipdg-discrete} by using
\(\widetilde S_h^m(\mathcal T_h)\) as both the trial and test space.
Table~\ref{tab:gcfe_gcife_coefficient_fit} reports the observed convergence
rates for the interface problem specified in
Subsection~\ref{sec:gcfe_dg_interface}, \yhred{with exact solution
\eqref{eq:interface-example-solution}, }computed on a sequence of
interface-unfitted meshes for the diffusion-coefficient pairs
\[
(\beta^-,\beta^+)
\in
\left\{
(1000,1),\ (10,1),\ (1,10),\ (1,1000)
\right\}.
\]
These pairs represent moderate and large coefficient jumps in both
directions. The domain \(\Omega\), outer boundary \(\Gamma_b\), interface
\(\Gamma_i\), and a representative unfitted mesh are shown in the right
panel of Figure~\ref{fig:gcfe_dg_flower_mesh}.
Figure~\ref{fig:gcfe_gcife_app_convergence} presents the \(L^2\)-norm and
broken \(H^1\)-seminorm errors versus \(h\) for \yhred{\(1\le m\le 4\)}, with
\(
(\beta^-,\beta^+)=(1,1000).
\)

The observed rates are consistent with the optimal \(O(h^{m+1})\) and
\(O(h^m)\) convergence orders in the \(L^2\) norm and broken
\(H^1\) seminorm, respectively, for all tested polynomial degrees and
coefficient pairs. These results indicate that the unified
SIPDG-GC-FE-GC-IFE method provides a viable high-order approach to
interface problems on curved domains using meshes independent of the
interface.

\begin{figure}[htbp]
    \centering
\captionsetup[subfigure]{font=normalfont,labelfont=normalfont,textfont=normalfont,skip=2pt}
    \begin{subfigure}[t]{0.49\textwidth}
        \centering
        \includegraphics[width=\linewidth]
        {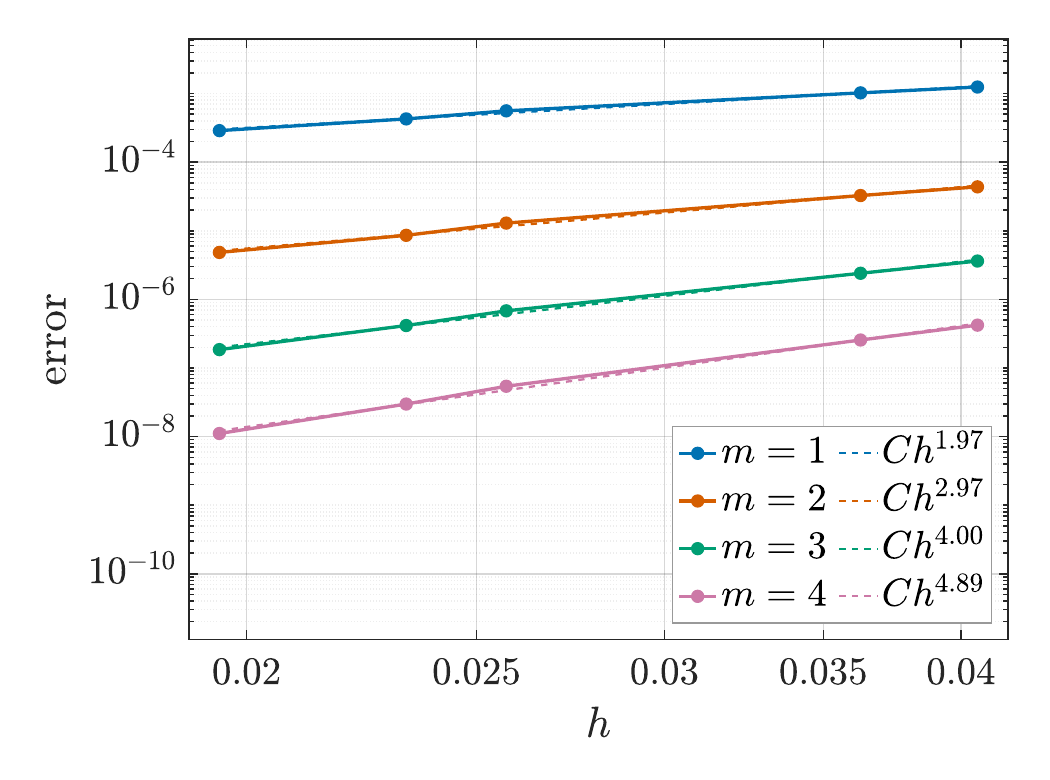}
        \caption{\(L^2\)-norm error.}
    \end{subfigure}\hspace{0.01\textwidth}
    \begin{subfigure}[t]{0.49\textwidth}
        \centering
        \includegraphics[width=\linewidth]
        {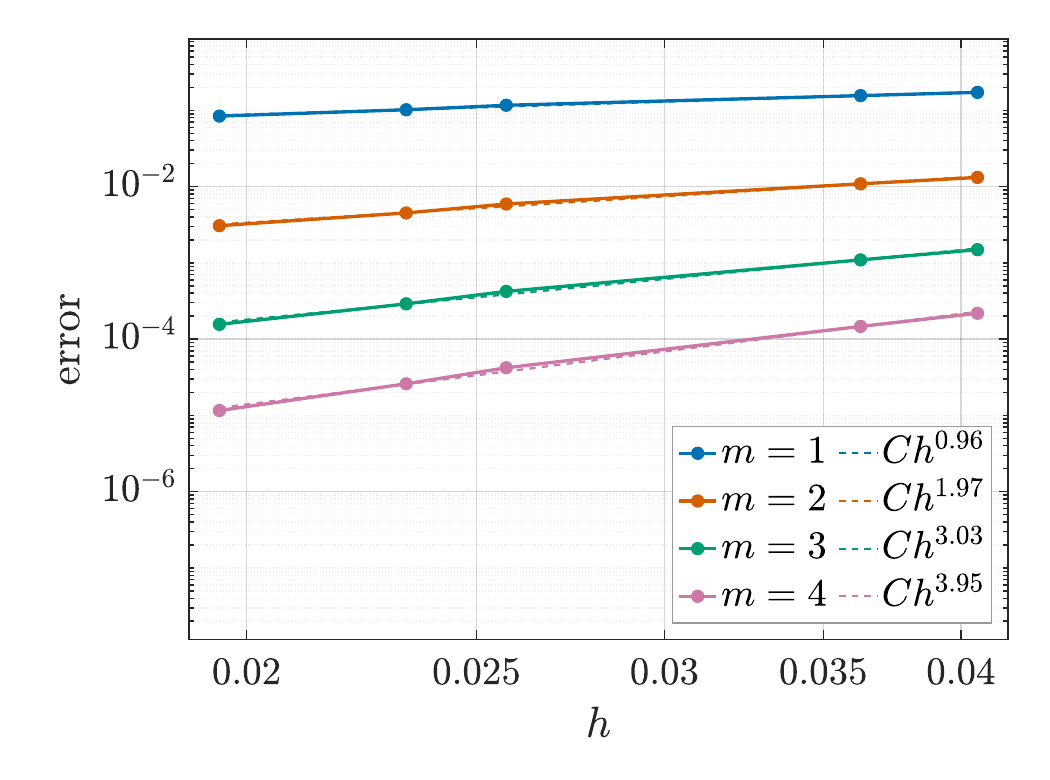}
        \caption{Broken \(H^1\)-seminorm error.}
    \end{subfigure}
    \caption{Errors of the unified SIPDG-GC-FE-GC-IFE method with a fitted
outer boundary and an unfitted internal interface, for
\((\beta^-,\beta^+)=(1,1000)\) and \yhred{\(1\le m\le 4\)}.}
    \label{fig:gcfe_gcife_app_convergence}
\end{figure}

\begin{table}[!htbp]
    \centering
    \caption{Convergence rates obtained by linear regression for the unified
        SIPDG-GC-FE-GC-IFE method.}
    \label{tab:gcfe_gcife_coefficient_fit}
     \small
    \setlength{\tabcolsep}{4pt}
    \begin{tabular}{c*{8}{c}}
\toprule
\(m\)&\multicolumn{4}{c}{\yhred{\(\|u-u_h\|_{L^2(\Omega)}\)}}&\multicolumn{4}{c}{\yhred{\(|u-u_h|_{H^1(\mathcal T_h)}\)}}\\
\cmidrule(lr){2-5}\cmidrule(lr){6-9}
&\((1000,1)\)&\((10,1)\)&\((1,10)\)&\((1,1000)\)&\((1000,1)\)&\((10,1)\)&\((1,10)\)&\((1,1000)\)\\
\midrule
1 & \yhred{1.95} & \yhred{1.95} & \yhred{1.95} & \yhred{1.97} & \yhred{1.00} & \yhred{1.00} & \yhred{1.00} & \yhred{0.96} \\
2 & \yhred{3.02} & \yhred{3.02} & \yhred{3.02} & \yhred{2.97} & 2.00 & 2.00 & 2.00 & \yhred{1.97} \\
3 & \yhred{4.00} & \yhred{4.00} & \yhred{4.00} & \yhred{4.00} & \yhred{3.03} & \yhred{3.03} & \yhred{3.03} & \yhred{3.03} \\
4 & \yhred{4.89} & \yhred{4.89} & \yhred{4.89} & \yhred{4.89} & 3.95 & 3.95 & 3.95 & \yhred{3.95} \\
\bottomrule
\end{tabular}
\end{table}

%

\section{Conclusions}

We developed an arbitrary-degree GC-FE framework for curved-boundary and
interface problems. By constructing polynomial spaces in Frenet
coordinates, the method retains the exact physical curves while producing
generally nonpolynomial shape functions in Cartesian coordinates. We
established optimal approximation, inverse, and trace estimates and proved
well-posedness and coefficient-robust weighted energy error bounds
for the SIPDG discretization on fitted meshes. Optimal $L^2$ estimates
follow under the stated weighted dual regularity assumption. The analysis
requires no separate geometric consistency estimates.
Thus, the framework provides a unified, geometry-conforming, high-order
treatment of curved boundaries and interfaces. Its analysis is carried out
here on interface-fitted meshes, while the numerical experiments indicate
that coupling GC-FE spaces with GC-IFE spaces retains optimal convergence
on interface-unfitted meshes. A rigorous error analysis of the coupled
GC-FE-GC-IFE method will be presented in a forthcoming paper.

%

         
	\bibliographystyle{abbrvnat}
\bibliography{reference_GCFE}

\begin{thebibliography}{40}
\providecommand{\natexlab}[1]{#1}
\providecommand{\url}[1]{\texttt{#1}}
\expandafter\ifx\csname urlstyle\endcsname\relax
  \providecommand{\doi}[1]{doi: #1}\else
  \providecommand{\doi}{doi: \begingroup \urlstyle{rm}\Url}\fi

\bibitem[Abate and Tovena(2012)]{2012AbateTovena}
M.~Abate and F.~Tovena.
\newblock \emph{Curves and surfaces}.
\newblock Springer Science \& Business Media, 2012.

\bibitem[Abbena et~al.(2006)Abbena, Salamon, and Gray]{gray_2006}
E.~Abbena, S.~Salamon, and A.~Gray.
\newblock \emph{Modern Differential Geometry of Curves and Surfaces with
  {M}athematica}.
\newblock Chapman and Hall/CRC, 3rd edition, 2006.

\bibitem[Adjerid and Moon(2019)]{2019AdjeridMoon}
S.~Adjerid and K.~Moon.
\newblock An immersed discontinuous {G}alerkin method for acoustic wave
  propagation in inhomogeneous media.
\newblock \emph{SIAM J. Sci. Comput.}, 41\penalty0 (1):\penalty0 A139--A162,
  2019.
\newblock ISSN 1064-8275.

\bibitem[Adjerid et~al.(2017)Adjerid, Guo, and Lin]{2017AdjeridGuoLin}
S.~Adjerid, R.~Guo, and T.~Lin.
\newblock High degree immersed finite element spaces by a least squares method.
\newblock \emph{Int. J. Numer. Anal. Model.}, 14\penalty0 (4-5):\penalty0
  604--626, 2017.
\newblock ISSN 1705-5105.

\bibitem[Adjerid et~al.(2024)Adjerid, Lin, and
  Meghaichi]{2024AdjeridLinMeghaichi}
S.~Adjerid, T.~Lin, and H.~Meghaichi.
\newblock A high order geometry conforming immersed finite element for elliptic
  interface problems.
\newblock \emph{Comput. Methods Appl. Mech. Engrg.}, 420:\penalty0 116703,
  2024.
\newblock ISSN 0045-7825,1879-2138.
\newblock \doi{10.1016/j.cma.2023.116703}.
\newblock URL \url{https://doi.org/10.1016/j.cma.2023.116703}.

\bibitem[Adjerid et~al.(2025)Adjerid, Lin, and
  Meghaichi]{2025AdjeridLinMeghaichi}
S.~Adjerid, T.~Lin, and H.~Meghaichi.
\newblock The {F}renet immersed finite element method for elliptic interface
  problems: an error analysis.
\newblock \emph{Comput. Methods Appl. Mech. Engrg.}, 438:\penalty0 117829,
  2025.
\newblock ISSN 0045-7825,1879-2138.
\newblock \doi{10.1016/j.cma.2025.117829}.
\newblock URL \url{https://doi.org/10.1016/j.cma.2025.117829}.

\bibitem[Adjerid et~al.(2026)Adjerid, Lin, and
  Meghaichi]{AdjeridLinMeghaichi2026}
S.~Adjerid, T.~Lin, and H.~Meghaichi.
\newblock Construction of basis functions for the geometry conforming immersed
  finite element method.
\newblock \emph{Numerical Algorithms}, pages DOI:
  10.1007/s11075--026--02424--6, 2026.

\bibitem[Burman and Zunino(2006)]{2006BurmanZunino}
E.~Burman and P.~Zunino.
\newblock A domain decomposition method based on weighted interior penalties
  for advection-diffusion-reaction problems.
\newblock \emph{SIAM J. Numer. Anal.}, 44\penalty0 (4):\penalty0 1612--1638,
  2006.

\bibitem[Burman et~al.(2015)Burman, Claus, Hansbo, Larson, and
  Massing]{2015BurmanClausHansboLarsonMassing}
E.~Burman, S.~Claus, P.~Hansbo, M.~G. Larson, and A.~Massing.
\newblock Cut{FEM}: discretizing geometry and partial differential equations.
\newblock \emph{Int. J. Numer. Methods Eng.}, 104\penalty0 (7):\penalty0
  472--501, 2015.

\bibitem[Burman et~al.(2025)Burman, Hansbo, Larson, and
  Zahedi]{2025BurmanHansboLarsonZahedi}
E.~Burman, P.~Hansbo, M.~G. Larson, and S.~Zahedi.
\newblock Cut finite element methods.
\newblock \emph{Acta Numer.}, 34:\penalty0 1--121, 2025.

\bibitem[Cai et~al.(2011)Cai, Ye, and Zhang]{2011CaiYeZhang}
Z.~Cai, X.~Ye, and S.~Zhang.
\newblock Discontinuous {G}alerkin finite element methods for interface
  problems: a priori and a posteriori error estimations.
\newblock \emph{SIAM J. Numer. Anal.}, 49\penalty0 (5):\penalty0 1761--1787,
  2011.

\bibitem[Cai et~al.(2017)Cai, He, and Zhang]{2017CaiHeZhang2}
Z.~Cai, C.~He, and S.~Zhang.
\newblock Discontinuous finite element methods for interface problems: robust a
  priori and a posteriori error estimates.
\newblock \emph{SIAM J. Numer. Anal.}, 55\penalty0 (1):\penalty0 400--418,
  2017.

\bibitem[Chen and Zhang(2021)]{2021ChenZhang}
Y.~Chen and X.~Zhang.
\newblock A {$P_2$}-{$P_1$} partially penalized immersed finite element method
  for {Stokes} interface problems.
\newblock \emph{Int. J. Numer. Anal. Model.}, 18\penalty0 (1):\penalty0
  120--141, 2021.

\bibitem[Chen and Zhang(2026)]{2026ChenZhang}
Y.~Chen and X.~Zhang.
\newblock An immersed {$C^0$} interior penalty finite element method for
  biharmonic interface problems.
\newblock \emph{J. Sci. Comput.}, 108:\penalty0 Paper No. 38, 2026.

\bibitem[Ciarlet(1978)]{1978Ciarlet}
P.~G. Ciarlet.
\newblock \emph{The finite element method for elliptic problems}.
\newblock North-Holland Publishing Co., Amsterdam-New York-Oxford, 1978.
\newblock ISBN 0-444-85028-7.
\newblock Studies in Mathematics and its Applications, Vol. 4.

\bibitem[Ciarlet and Raviart(1972{\natexlab{a}})]{1972CiarletRaviart}
P.~G. Ciarlet and P.-A. Raviart.
\newblock Interpolation theory over curved elements, with applications to
  finite element methods.
\newblock \emph{Comput. Methods Appl. Mech. Engrg.}, 1\penalty0 (2):\penalty0
  217--249, 1972{\natexlab{a}}.

\bibitem[Ciarlet and Raviart(1972{\natexlab{b}})]{1972CiarletRaviart2}
P.~G. Ciarlet and P.-A. Raviart.
\newblock The combined effect of curved boundaries and numerical integration in
  isoparametric finite element methods.
\newblock In \emph{The mathematical foundations of the finite element method
  with applications to partial differential equations}, pages 409--474.
  Academic Press, 1972{\natexlab{b}}.

\bibitem[Dryja(2003)]{2003Dryja}
M.~Dryja.
\newblock On discontinuous {G}alerkin methods for elliptic problems with
  discontinuous coefficients.
\newblock \emph{Comput. Methods Appl. Math.}, 3\penalty0 (1):\penalty0 76--85,
  2003.

\bibitem[Ern et~al.(2009)Ern, Stephansen, and Zunino]{2009ErnStephansenZunino}
A.~Ern, A.~F. Stephansen, and P.~Zunino.
\newblock A discontinuous {G}alerkin method with weighted averages for
  advection--diffusion equations with locally small and anisotropic
  diffusivity.
\newblock \emph{IMA J. Numer. Anal.}, 29\penalty0 (2):\penalty0 235--256, 2009.

\bibitem[Gordon and Hall(1973{\natexlab{a}})]{GordonHall1973b}
W.~J. Gordon and C.~A. Hall.
\newblock Transfinite element methods: Blending-function interpolation over
  arbitrary curved element domains.
\newblock \emph{Numer. Math.}, 21:\penalty0 109--129, 1973{\natexlab{a}}.

\bibitem[Gordon and Hall(1973{\natexlab{b}})]{gordon1973construction}
W.~J. Gordon and C.~A. Hall.
\newblock Construction of curvilinear co-ordinate systems and applications to
  mesh generation.
\newblock \emph{Int. J. Numer. Methods Eng.}, 7\penalty0 (4):\penalty0
  461--477, 1973{\natexlab{b}}.

\bibitem[Guo and Lin(2019)]{2019GuoLin2}
R.~Guo and T.~Lin.
\newblock A higher degree immersed finite element method based on a {C}auchy
  extension for elliptic interface problems.
\newblock \emph{SIAM J. Numer. Anal.}, 57\penalty0 (4):\penalty0 1545--1573,
  2019.
\newblock ISSN 0036-1429.
\newblock \doi{10.1137/18M121318X}.
\newblock URL \url{https://doi.org/10.1137/18M121318X}.

\bibitem[Guo and Lin(2020)]{2020GuoLin}
R.~Guo and T.~Lin.
\newblock An immersed finite element method for elliptic interface problems in
  three dimensions.
\newblock \emph{J. Comput. Phys.}, 414\penalty0 (1):\penalty0 109478, 2020.

\bibitem[Guo and Zhang(2021)]{2021GuoZhang}
R.~Guo and X.~Zhang.
\newblock Solving three-dimensional interface problems with immersed finite
  elements: {A}-priori error analysis.
\newblock \emph{J. Comput. Phys.}, 441:\penalty0 110445, 2021.
\newblock ISSN 0021-9991.

\bibitem[Guo et~al.(2020)Guo, Lin, and Lin]{2020GuoLinLin}
R.~Guo, T.~Lin, and Y.~Lin.
\newblock Error estimates for a partially penalized immersed finite element
  method for elasticity interface problems.
\newblock \emph{ESAIM Math. Model. Numer. Anal.}, 54\penalty0 (1):\penalty0
  1--24, 2020.
\newblock ISSN 0764-583X.

\bibitem[He et~al.(2008)He, Lin, and Lin]{2008HeLinLin}
X.~He, T.~Lin, and Y.~Lin.
\newblock Approximation capability of a bilinear immersed finite element space.
\newblock \emph{Numer. Methods Partial Differential Equations}, 24\penalty0
  (5):\penalty0 1265--1300, 2008.
\newblock ISSN 0749-159X.
\newblock \doi{10.1002/num.20318}.
\newblock URL \url{http://dx.doi.org/10.1002/num.20318}.

\bibitem[Hesthaven and Warburton(2008)]{2008HesthavenWarburton}
J.~S. Hesthaven and T.~Warburton.
\newblock \emph{Nodal discontinuous {G}alerkin methods}, volume~54 of
  \emph{Texts in Applied Mathematics}.
\newblock Springer, New York, 2008.
\newblock ISBN 978-0-387-72065-4.
\newblock \doi{10.1007/978-0-387-72067-8}.
\newblock URL \url{https://doi.org/10.1007/978-0-387-72067-8}.
\newblock Algorithms, analysis, and applications.

\bibitem[Hughes and Sangalli(2017)]{HughesSangalli2017}
T.~J.~R. Hughes and G.~Sangalli.
\newblock {Mathematics of Isogeometric Analysis: A Conspectus}.
\newblock In E.~Stein, R.~de~Borst, and T.~J.~R. Hughes, editors,
  \emph{Encyclopedia of Computational Mechanics Second Edition}, pages 1--40.
  John Wiley \& Sons, Ltd., 2017.
\newblock \doi{10.1002/9781119176817.ecm2100}.

\bibitem[Hughes et~al.(2005)Hughes, Cottrell, and Bazilevs]{Hughes2005}
T.~J.~R. Hughes, J.~A. Cottrell, and Y.~Bazilevs.
\newblock Isogeometric analysis: {CAD}, finite elements, {NURBS}, exact
  geometry and mesh refinement.
\newblock \emph{Comput. Methods Appl. Mech. Engrg.}, 194\penalty0
  (39--41):\penalty0 4135--4195, 2005.
\newblock \doi{10.1016/j.cma.2004.10.008}.

\bibitem[Lehrenfeld and Reusken(2018)]{2018LehrenfeldReusken}
C.~Lehrenfeld and A.~Reusken.
\newblock Analysis of a high-order unfitted finite element method for elliptic
  interface problems.
\newblock \emph{IMA J. Numer. Anal.}, 38\penalty0 (3):\penalty0 1351--1387,
  2018.

\bibitem[Lenoir(1986)]{1986Lenoir}
M.~Lenoir.
\newblock Optimal isoparametric finite elements and error estimates for domains
  involving curved boundaries.
\newblock \emph{SIAM J. Numer. Anal.}, 23\penalty0 (3):\penalty0 562--580,
  1986.
\newblock ISSN 0036-1429.
\newblock \doi{10.1137/0723036}.
\newblock URL \url{https://doi.org/10.1137/0723036}.

\bibitem[Li et~al.(2010)Li, Melenk, Wohlmuth, and Zou]{LiMelenkWohlmuthZou2010}
J.~Li, J.~M. Melenk, B.~Wohlmuth, and J.~Zou.
\newblock Optimal a priori estimates for higher order finite elements for
  elliptic interface problems.
\newblock \emph{Appl. Numer. Math.}, 60\penalty0 (1--2):\penalty0 19--37, 2010.

\bibitem[Li(1998)]{Li1998}
Z.~Li.
\newblock The immersed interface method using a finite element formulation.
\newblock \emph{Appl. Numer. Math.}, 27\penalty0 (3):\penalty0 253--267, 1998.

\bibitem[Li et~al.(2003)Li, Lin, and Wu]{LiLinWu2003}
Z.~Li, T.~Lin, and X.~Wu.
\newblock New {Cartesian} grid methods for interface problems using the finite
  element formulation.
\newblock \emph{Numer. Math.}, 96\penalty0 (1):\penalty0 61--98, 2003.

\bibitem[Lin et~al.(2013)Lin, Sheen, and Zhang]{2013LinSheenZhang}
T.~Lin, D.~Sheen, and X.~Zhang.
\newblock A locking-free immersed finite element method for planar elasticity
  interface problems.
\newblock \emph{J. Comput. Phys.}, 247:\penalty0 228--247, 2013.

\bibitem[Lin et~al.(2015)Lin, Lin, and Zhang]{LinLinZhang2015}
T.~Lin, Y.~Lin, and X.~Zhang.
\newblock Partially penalized immersed finite element methods for elliptic
  interface problems.
\newblock \emph{SIAM J. Numer. Anal.}, 53\penalty0 (2):\penalty0 1121--1144,
  2015.

\bibitem[Lin et~al.(2026)Lin, Zhang, and Lin]{2026LinLinZhang}
Y.~Lin, X.~Zhang, and T.~Lin.
\newblock Frenet immersed finite element spaces on triangular meshes.
\newblock \emph{Int. J. Numer. Anal. Model.}, page in press, 2026.

\bibitem[O'Neill(2006)]{ONeill_DiffGeom_2010}
B.~O'Neill.
\newblock \emph{Elementary differential geometry}.
\newblock Academic Press, second edition, 2006.

\bibitem[Sevilla et~al.(2008)Sevilla, Fern{\'a}ndez-M{\'e}ndez, and
  Huerta]{sevilla2008nurbs}
R.~Sevilla, S.~Fern{\'a}ndez-M{\'e}ndez, and A.~Huerta.
\newblock {NURBS}-enhanced finite element method ({NEFEM}).
\newblock \emph{Int. J. Numer. Methods Eng.}, 76\penalty0 (1):\penalty0 56--83,
  2008.

\bibitem[Strang and Fix(1973)]{1973StrangFix}
G.~Strang and G.~J. Fix.
\newblock \emph{An analysis of the finite element method}.
\newblock Prentice-Hall Series in Automatic Computation. Prentice-Hall, Inc.,
  Englewood Cliffs, NJ, 1973.

\end{thebibliography}
	
\end{document}